\documentclass{article}
\usepackage[english]{babel}
\usepackage[latin1]{inputenc}
\usepackage{graphicx}
\usepackage{amssymb}
\usepackage{amsmath}
\usepackage{amsthm}
\usepackage{enumerate}
\usepackage{color,xcolor}
\newtheorem{theorem}{Theorem}[section]
\newtheorem{lemma}[theorem]{Lemma}
\newtheorem{corollary}[theorem]{Corollary}
\newtheorem{proposition}[theorem]{Proposition}
\theoremstyle{definition}
\newtheorem{remark}[theorem]{Remark}
\newtheorem{example}[theorem]{Example}
\newtheorem{examples}[theorem]{Examples}
\newtheorem{definition}[theorem]{Definition}
\numberwithin{equation}{section}

\allowdisplaybreaks[1]

\numberwithin{equation}{section}

\newcommand{\nd}{\red{\ensuremath d}} 
\renewcommand{\nd}{{\ensuremath d}} 

\newcommand{\B}{\ensuremath{B^s_{p,q}}}

\newcommand{\ud}{\ensuremath{\mathrm{d}}}

\newcommand{\dint}{\ensuremath{\mathrm{d}}}
\newcommand{\bit}{\begin{itemize}}
\newcommand{\eit}{\end{itemize}}
\newcommand{\beq}{\begin{equation}}
\newcommand{\eeq}{\end{equation}}
\newcommand{\supp}{\mathrm{supp}\,}

\def\ls{\lesssim}
\newcommand{\Lloc}{L_1^{\mathrm{loc}}}
\newcommand{\nat}{\ensuremath{\mathbb{N}}}
\newcommand{\no}{\ensuremath{\nat_0}}
\newcommand{\rr}{\ensuremath{{\mathbb R}}}
\newcommand{\rd}{\ensuremath{{\mathbb R}^\nd}}

\newcommand{\zz}{\ensuremath{{\mathbb Z}}}
\newcommand{\zd}{\ensuremath{\mathbb{Z}^\nd}}

\newcommand{\M}{{\mathcal M}_{\varphi,p}}
\newcommand{\Mu}{{\mathcal M}_{u,p}}
\newcommand{\Me}{{\mathcal M}_{u_1,p_1}}
\newcommand{\Mz}{{\mathcal M}_{u_2,p_2}}

\newcommand{\MB}{{\mathcal N}^{s}_{\varphi,p,q}}

\newcommand{\MF}{{\mathcal E}^{s}_{\varphi,p,q}}
\newcommand{\Gp}{{\mathcal G}_p}
\newcommand{\n}{{n}^{s}_{\varphi,p,q}}

\definecolor{verde}{RGB}{0,205,0}

\newcommand{\whole}[1]{\ensuremath \lfloor #1 \rfloor}
\newcommand{\up}[1]{\ensuremath  \lceil #1 \rceil}
\newcommand{\blue}[1]{{\color{blue}#1}}
\newcommand{\egv}[1]{\ensuremath{{\mathcal E}\rule[-0.7ex]{0em}{2.4ex}_{\!\mathsf G}^{#1}}}
\newcommand{\egX}{\egv{X}}
\newcommand{\envg}{\ensuremath{{\mathfrak E}\rule[-0.5ex]{0em}{2.3ex}_{\!\mathsf G}}}
\newcommand{\uGv}[1]{\ensuremath{u_{\mathsf{G}}^{#1}}}
\newcommand{\uGindex}{\uGv{X}}
\newcommand{\fundfn}[1]{\varphi\rule[-0.6ex]{0ex}{2ex}_{#1}}
\newcommand{\charfn}[1]{\chi\rule[-1ex]{0ex}{2ex}_{#1}}
\newcommand{\ds}{\displaystyle}

\begin{document}

\title{
{On a bridge connecting Lebesgue and Morrey spaces in view of their growth properties}
}
\date{\today}
\author{ Dorothee D. Haroske\footnotemark[1],
Susana D. Moura\footnotemark[2],   and Leszek Skrzypczak\footnotemark[3]}
\maketitle

\footnotetext[1]{{The author was partially supported by the German Research Foundation (DFG), Grant no. Ha 2794/8-1.}}
\footnotetext[2]{{The author was partially supported by the Center for Mathematics of the University of Coimbra-UIDB/00324/2020, funded by the Portuguese Government through FCT/MCTES.}}
\footnotetext[3]{\blue{ ??}}

\begin{abstract} 
We study unboundedness properties of functions belonging to  generalised Morrey spaces $\M(\rd)$ and generalised Besov-Morrey spaces $\MB(\rd)$ by means of growth envelopes. 
For the generalised Morrey spaces  we arrive at  the same three possible cases as for classical Morrey spaces $\mathcal{M}_{u,p}(\rd)$, i.e., boundedness, the $L_p$-behaviour  or the proper Morrey behaviour for $p<u$, but now those cases are characterised in terms of the limit of $\varphi(t)$ and $t^{-\nd/p} \varphi(t)$ as $t \to 0^+$ and $t\to\infty$, respectively.  For the generalised Besov-Morrey spaces  the limit  of  $t^{-\nd/p} \varphi(t)$ as $t \to 0^+$ also plays a r\^ole and, once more, we are able to extend to this generalised spaces the known results for classical Besov-Morrey spaces, although some cases are not completely solved. In this context we can completely characterise the situation when $\MB(\rd)$ consists of essentially bounded functions only, and when it contains regular distributions only.
\end{abstract}

\medskip

{\bfseries MSC (2010)}: \smallskip 46E35\\


\medskip\medskip

\section{Introduction}
In this paper we study smoothness function spaces built upon generalised Morrey spaces $\M(\rd)$, $0<p<\infty$, $\varphi:(0,\infty)\rightarrow [0,\infty)$. The generalised version of Morrey spaces $\mathcal{M}_{u,p}(\rd)$, $0 < p \le u < \infty$, {was} introduced by  Mizuhara \cite{mi} and Nakai \cite{Nak94}, and later applied successfully   to PDEs, cf. \cite{FHS,KMR,WNTZ,ZJSZ,KNS}.  We refer to \cite{Saw18} for further information about the spaces and the historical remarks.

In recent years also smoothness function spaces built upon Morrey spaces $\mathcal{M}_{u,p}(\rd)$ were investigated,  in particular Besov-Morrey spaces  $\mathcal{N}^s_{u,p,q}(\rd)$, $0 < p \le u < \infty$, $0 < q \le \infty$, $s \in \rr$.  They were introduced by Kozono and Yamazaki in \cite{KY} and used by them to study Navier-Stokes equations. Further applications of the spaces to PDEs can be found e.g. in the  papers written by Mazzucato \cite{Maz03}, by  Ferreira and Postigo \cite{FP} or by Yang, Fu, and Sun \cite{YFS}. 

Now we study the Besov spaces  $\MB(\rd)$ built upon generalised Morrey spaces. These spaces were introduced and studied by Nakamura, Noi and Sawano  \cite{NNS16}, cf. also \cite{AGNS}. In particular, they proved the  atomic decomposition theorem for the spaces. Recently Izuki and Noi investigated
spaces on domains in \cite{IN19}. The generalised Besov-Morrey spaces cover Besov-Morrey spaces and local Besov-Morrey spaces considered by Triebel \cite{Tri13} as special cases.

The paper relies on our recent outcome \cite{hms22} where we studied the spaces of type $\MB(\rd)$ in some detail, in particular, their embeddings and wavelet decomposition. Now we concentrate on the two questions, when such spaces contain regular distributions only, and -- in that case -- what unboundedness of the corresponding functions is possible. Our answer to the first question is almost complete and can be found in Theorem~\ref{emb-L1loc} below. Clearly this depends -- as in the case of Besov spaces $\B(\rd)$ -- 
upon the smoothness $s$, the integrability $p$, and in limiting cases also on the fine index $q$. It is remarkable that the expected influence of the Morrey parameter function $\varphi$ is only visible in case of $0<p<1$ and $\lim_{t\to 0^+} \varphi(t)=0$, while in all other cases the conditions for $\MB(\rd)\subset \Lloc(\rd)$ are literally the same as for $\B(\rd)\subset \Lloc(\rd)$, we refer to  Theorem~\ref{emb-L1loc} for details.

Somewhat supplementary we present the complete characterisation when the Dirac $\delta$-distribution belongs to the smoothness spaces of Besov-Morrey type $\MB(\rd)$, that is,
\[
	\delta\in\MB(\rd) \quad\text{if, and only if,}\quad  \left\{2^{j(s+\nd)}  \varphi(2^{-j})\right\}_{j\in\no}\in\ell_q , 
\]
where  $s\in\rr$, $0<p<\infty$, $0<q\leq \infty$, and $\varphi\in \Gp$; see Lemma~\ref{delta-gen-N} below.
        
Concerning the second question, this is characterised in terms of the growth envelope function. More precisely, we would like to understand the `quality' of the unboundedness in situations when $\M(\rd)\not\hookrightarrow L_\infty(\rd)$ or $\MB(\rd)\not\hookrightarrow L_\infty(\rd)$. We use the tool of the growth envelope function, that is, we study the behaviour  
\[
 \egX(t) \sim \sup\left\{ f^\ast(t)  : \ \|f|X\| \leq 1\right\} ,\quad
t>0,
\]
where $f^\ast$ denotes the non-increasing rearrangement of $f$, and $X$ stands for some function space, like $\M(\rd)$ or $\MB(\rd)$ in our setting. This approach turned out to be a rather powerful one, admitting a number of interesting applications. These concepts were introduced in \cite{T-func}, \cite{Ha-crc}, where the latter book also contains a survey (concerning extensions and more general approaches) as well as applications and further references. While in the easiest example, $X=L_p(\rd)$, the outcome reads as $\egv{L_p(\rd)}(t) \sim t^{-\frac1p}$, $t>0$, we found in \cite{HaSM3} that $\egv{\mathcal{M}_{u,p}(\rd)} = \infty$, $t>0$, whenever $u>p$. In \cite{hms16} we refined and extended this result to all admitted spaces $\mathcal{N}^s_{u,p,q}(\rd)$. In that sense the present paper can be understood as a continuation and generalisation of our studies in \cite{HaSM3,hms16} here. We compare it also with our findings in \cite{HSS-morrey,HMSS-morrey} in view of local behaviour. Now, dealing with the spaces $\M(\rd)$, the result depends on the local and global behaviour of the function $\varphi$. We observe some new phenomenon which can be roughly summarised as
 \[
  \begin{cases}
    \egv{\M(\rd)}(t) \ \sim \ 1 & \text{if, and only if,}\qquad \lim_{t\rightarrow 0^+} \varphi(t)>0, \\ 
\egv{\M(\rd)}(t) \ \sim \ t^{-1/p} \quad & \text{if, and only if,}\qquad \lim_{t\rightarrow 0^+} \varphi(t)=0 \quad\text{and}\quad \lim_{t\rightarrow \infty} t^{-\nd/p}\varphi(t)>0, \\
    \egv{\M(\rd)}(t) \ = \ \infty & \text{if, and only if,}\qquad \lim_{t\rightarrow 0^+} \varphi(t)=0 \quad\text{and}\quad \lim_{t\rightarrow \infty} t^{-\nd/p}\varphi(t)=0.
    \end{cases}
  \]
  In view of the assumption $\varphi\in\Gp$ this is a complete description of all possible cases. In other words, it means that for arbitrary admitted $\varphi$, essentially all situations are already covered by our previous studies: we have boundedness (first line), the $L_p$-behaviour (second line), or the (proper, classical) Morrey behaviour (last line). This seems the first such observation and was not to be expected from the very beginning. We refer to Corollaries~\ref{egM-coll}, \ref{egM-finite}, and  Theorem~\ref{egM-LS} below for details. The situation is more complicated, but to some extent similar for smoothness spaces $\MB(\rd)$ of Besov-Morrey type, we refer to {Propositions~\ref{P-case-1}, \ref{P-case-2}, \ref{P-case-3} and \ref{P-case-4}}, though not all phenomena are completely understood yet.  Besides the limits
$\lim_{t\rightarrow 0^+} \varphi(t)$ and $\lim_{t\rightarrow \infty} t^{-\nd/p}\varphi(t)$, 
already appearing when studying the generalised Morrey spaces, also the limit 
$\ds \lim_{t\rightarrow 0^+} \varphi(t) t^{-\frac{\nd}{p}} $ 
plays a r\^ole regarding the smoothness spaces, where now, in addition, the interplay of the smoothness parameter $s$ of $\MB(\rd)$ with the function $\varphi$ is important.

  The paper is organised as follows. In Section~\ref{prelim} we first briefly collect general notation before we then concentrate on the setting of generalised Morrey spaces $\M(\rd)$, including the characterisation of the possible unboundedness in terms of the growth envelope function. In Section~\ref{gen-Nphi} we deal with the concept of generalised Besov-Morrey spaces $\MB(\rd)$ and first approach the question, mentioned above, to characterise $\MB(\rd)\subset \Lloc(\rd)$ or $\delta\in\MB(\rd)$. Afterwards we study the growth envelope function of $\MB(\rd)$. In our arguments we rely on our preceding results for the cases $\mathcal{M}_{u,p}(\rd)$ and $\mathcal{N}_{u,p,q}^s(\rd)$, including, in particular $L_p(\rd)$ and $\B(\rd)$, as well as on the wavelet and atomic decomposition of $\MB(\rd)$. Occasionally we exemplify the outcome with some example functions $\varphi$.
  
Let us mention that the first motivation to study such questions came from the master's thesis by M.~Bauer \cite{bauer} supervised by the first author.

\section{Generalised Morrey spaces}\label{prelim}

First we fix some notation. By $\nat$ we denote the {set of natural numbers},
by $\no$ the set $\nat \cup \{0\}$,  and by $\zd$ the {set of all lattice points
in $\rd$ having integer components}. Let $\no^\nd$, where $\nd\in\nat$, be the set of all multi-indices, $\alpha = (\alpha_1, \ldots,\alpha_\nd)$ with $\alpha_j\in\no$ and $|\alpha| := \sum_{j=1}^\nd \alpha_j$. If $x=(x_1,\ldots,x_\nd)\in\rd$ and $\alpha = (\alpha_1, \ldots,\alpha_\nd)\in\no^\nd$, then we put $x^\alpha := x_1^{\alpha_1} \cdots x_\nd^{\alpha_\nd}$. 
For $a\in\rr$, let   $\whole{a}:=\max\{k\in\zz: k\leq a\}$, $\up{a} = \min\{k\in \zz:\; k\ge a \}$, and $a_{+}:=\max(a,0)$.
Given any $u\in (0,\infty]$, it will be denoted by $u'$ the number, possibly $\infty$, defined by the expression  $\frac{1}{u'}=(1-\frac{1}{u})_+$; in particular when $1\leq u\leq \infty$, $u'$ is the same as the conjugate exponent defined through $\frac{1}{u}+ \frac{1}{u'}=1$.
All unimportant positive constants will be denoted by $C$, 
occasionally the same letter $C$ is used to denote different constants  in the same chain of inequalities.
 By the notation $A \ls B$, we mean that there exists a positive constant $c$ such that
 $A \le c \,B$, whereas  the symbol $A \sim B$ stands for $A \ls B \ls A$.
 We denote by
 $|\cdot|$ the Lebesgue measure when applied to measurable subsets of $\rd$.
For each cube $Q\subset \rd $ we denote  its side length by $ \ell(Q)$, and, for $a\in (0,\infty)$, we denote by $aQ$ the cube concentric with $Q$ having the side length $a\ell(Q)$. For $x\in\rd$ and $r \in (0, \infty)$  we denote by $Q(x, r)$ the compact cube centred at $x$ with side length $r$, whose sides are parallel to the axes of coordinates. We write simply $Q(r)=Q(0,r)$ when $x=0$.
By $\mathcal{Q}$ we denote the collection of all {dyadic cubes} in $\rd$, namely, $\mathcal{Q}:= \{Q_{j,k}:= 2^{-j}([0,1)^\nd+k):\ j\in\zz,\ k\in\zd\}$. 
%

Given two (quasi-)Banach spaces $X$ and $Y$, we write $X\hookrightarrow Y$
if $X\subset Y$ and the natural embedding of $X$ into $Y$ is continuous.

Recall first  that the classical  Morrey space
${\mathcal M}_{u,p}(\rd)$, $0<p\leq u<\infty $, is defined to be the set of all
  locally $p$-integrable functions $f\in L_p^{\mathrm{loc}}(\rd)$  such that
$$
\|f \mid {\mathcal M}_{u,p}(\rd)\| :=\, \sup_{Q\in \mathcal{Q}} |Q|^{\frac{1}{u}}
\left( {\frac{1}{|Q|}}\int_{Q} |f(y)|^p \dint y \right)^{\frac{1}{p}}\, <\, \infty\, .
$$

In this paper we consider generalised Morrey spaces where the parameter $u$ is replaced by a function $\varphi$ according to the following definition. 

\begin{definition}
Let $0<p<\infty$ and $\varphi:(0,\infty)\rightarrow [0,\infty)$ be a function which does not satisfy $\varphi\equiv 0$. Then $\M(\rd)$ is the set of all 
locally $p$-integrable functions $f\in L_p^{\mathrm{loc}}(\rd)$ for which 
\begin{equation*} 
{\|f\mid \M(\rd)\|_\star:=\sup_{x\in\rd, r>0} \varphi(r) 
\biggl(\frac{1}{|Q(x,r)|}\int_{Q(x,r)} |f(y)|^p \dint y\biggr)^{\frac{1}{p}} \, <\, \infty\, . }
\end{equation*}
\end{definition}

\begin{remark} \label{rmk1}
The above definition goes back to \cite{Nak94}.
When $\varphi(t)=t^{\frac{\nd}{u}}$ for $t>0$ and $0<p\leq u<\infty$, then $\M(\rd)$ coincides with ${\mathcal M}_{u,p}(\rd)$, which in turn recovers the Lebesgue space $L_p(\rd)$ when $u=p$. 
Note that for $\varphi_0\equiv 1$ (which would correspond to $u=\infty$) we obtain 
\begin{equation}\label{M1p}
{\mathcal M}_{\varphi_0,p}(\rd) = L_\infty(\rd),\quad 0<p<\infty,\quad \varphi_0\equiv 1,
\end{equation}
due to Lebesgue's differentiation theorem.  

When $\varphi(t)=t^{-\sigma} \chi_{(0,1)}(t)$ where  $-\frac{\nd}{p}\leq \sigma <0$, then $\M(\rd)$ coincides with the local Morrey spaces $\mathcal{L}^{\sigma}_p(\rd)$ introduced by H.~Triebel in \cite{Tri11}, cf. also \cite[Section~1.3.4]{Tri13}. If $ \sigma=-\frac{\nd}{p}$, then the space is a uniform Lebesgue space $\mathcal{L}_p(\rd)$.
\end{remark}

\bigskip 

For $\M(\rd)$ it is usually required that $\varphi\in \Gp$, where $ \Gp$ 
is the set of all nondecreasing functions $\varphi:(0,\infty)\rightarrow (0,\infty)$ such that $\varphi(t)t^{-\nd/p}$ is a nonincreasing function, i.e.,
\begin{equation}\label{Gp-def}
1\leq \frac{\varphi(r)}{\varphi(t)}\leq \left(\frac{r}{t}\right)^{\nd/p}, \quad 0<t\leq r<\infty.
\end{equation}
A justification for the use of the class $\Gp$ comes from the  lemma below, cf. e.g. \cite[Lem. 2.2]{NNS16}.  One can easily check that $\mathcal{G}_{p_2}\subset \mathcal{G}_{p_1}$ if $0<p_1\le p_2<\infty$. 


  \begin{remark}
   { {If $\varphi\in \Gp$,  then it enjoys a doubling property, i.e.,  $\varphi(r)\le \varphi(2r) \le 2^{\nd/p}\varphi(r)$, $r\in (0,\infty)$. In consequence we can define an equivalent quasi-norm in  $\M(\rd)$ by taking the supremum over the collection of all dyadic cubes. Later on we  always assume that $\varphi\in \Gp$ and we use the quasi-norm 
   	\begin{equation} \label{Morrey-norm}
   		\|f\mid \M(\rd)\|:=\sup_{Q\in\mathcal{Q}}\varphi \bigl(\ell(Q)\bigr)\biggl(\frac{1}{|Q|}\int_Q |f(y)|^p \dint y\biggr)^{\frac{1}{p}} . 
   	\end{equation}
    This choice is natural and covers all interesting cases, cf. Lemma~\ref{eqvarphi} below.  A similar argument shows that in the definition of $\|\cdot\mid \M(\rd)\|_\star$ balls instead of  cubes can be taken. This  change leads once more to  equivalent quasi-norms. }
   	   		
   	We shall usually assume in the sequel  that $\varphi(1)=1$. This simplifies the notation but does not limit the generality of considerations. Note that for a function $\varphi\in\Gp$ with $\varphi(1)=1$, then \eqref{Gp-def} implies } that
\begin{align}\label{Gp-tlarge}
1\leq \varphi(t)\leq t^{\frac{\nd}{p}},\quad t\geq 1, & \quad \text{in particular},\quad 
  \varphi(t) t^{-\frac{\nd}{p}} \leq 1\quad \text{for}\qquad t\to\infty,
  \intertext{and}\label{Gp-tsmall}
  t^{\frac{\nd}{p}}\leq \varphi(t)\leq 1,\quad 0<t\leq 1.&
\end{align}   
Together with the required monotonicity conditions for $\varphi\in\Gp$  this implies the existence of both limits
\begin{equation}\label{lim-exist}
  \lim_{t\to\infty} \varphi(t) t^{-\frac{\nd}{p}} \in [0,1] \qquad \text{and}\qquad \lim_{t\to 0^+} \varphi(t) \in [0,1].
  \end{equation}
Later it will turn out that the special cases when one or both limits equal $0$ play an important r\^ole. 
\end{remark}

\begin{lemma}[\cite{NNS16,Saw18}]\label{eqvarphi} 
Let $0<p<\infty$ and $\varphi:(0,\infty)\rightarrow [0,\infty)$ be a function satisfying $\varphi(t_0)\neq 0$ for some $t_0>0$. 
\begin{itemize}
\item[{\upshape\bfseries (i)}] Then $\M(\rd)\neq \{0\}$ if, and only if, 
$$
\sup_{t>0} \varphi(t)  \min (t^{-\frac{\nd}{p}},1) < \infty.
$$
\item[{\upshape\bfseries (ii)}] Assume \ $\displaystyle\sup_{t>0} \varphi(t)  \min (t^{-\frac{\nd}{p}},1) < \infty$. Then there exists $\varphi^*\in\Gp$ such that
$$
\M(\rd)={\mathcal M}_{\varphi^*,p}(\rd)
$$
in the sense of equivalent (quasi-)norms.
 \end{itemize}
\end{lemma}

  \begin{remark}\label{R-emb-Mphi}
    In \cite[Thm.~3.3]{GHLM17} it is shown that for $1\leq p_2\leq p_1<\infty$, $\varphi_i\in \mathcal{G}_{p_i}$, $i=1,2$, then
    \begin{equation} \label{emb-Mphi}
      {\mathcal M}_{\varphi_1, p_1}(\rd) \hookrightarrow {\mathcal M}_{\varphi_2, p_2}(\rd)
      \end{equation}
    if, and only if, there exists some $C>0$ such that for all $t>0$, $\varphi_1(t)\geq C \varphi_2(t)$. 
   The argument can be immediately extended to $0<p_2\leq p_1<\infty$, cf. \cite[Cor.~30]{FHS-MS-2}.
   
 In case of $\varphi_i(t) =t^{\nd/u_i}$, $0<p_i\leq u_i<\infty$, $i=1,2$, it is well-known that
\begin{equation}\label{emb-Mu}
\Me(\rd) \hookrightarrow \Mz(\rd)\qquad\text{if, and only if,}\qquad
p_2\leq p_1\leq u_1=u_2, 
\end{equation}
cf. \cite{piccinini-1} and \cite{rosenthal}. This can be observed from \eqref{emb-Mphi} { since} 
\[
\varphi_1(t)\geq C \varphi_2(t) \quad\text{{means}}\quad t^{\frac{\nd}{u_1}} \geq C t^{\frac{\nd}{u_2}},\quad t>0,
\]
which results in $u_1=u_2$.

Let $p\in (0,\infty)$, $\varphi\in \Gp$, and $0<r\leq p$. Then \eqref{emb-Mphi} applied to $\varphi_1=\varphi$, $\varphi_2(t)=t^{\nd/r}$ results in
\begin{equation*}
  \M(\rd)\hookrightarrow L_r(\rd) \quad\text{if, and only if,}\quad \varphi(t) \geq C t^{\frac{\nd}{r}},\quad t>0.
  \end{equation*}
In connection with \eqref{Gp-tlarge} and \eqref{Gp-tsmall} this leads to $r=p$, that is,
\begin{equation}\label{Mphi-Lr}
 \M(\rd)\hookrightarrow L_r(\rd) \quad\text{if, and only if,}\quad r=p \quad \text{and}\quad \varphi(t) \sim t^{\nd/p},\quad t>1.
\end{equation}

Finally, if there exists some $u\geq p$ such that $\varphi(t) \leq c t^{\nd/u}$, $t>0$, then
\begin{equation}\label{emb-Mu-Mphi}
  \mathcal{M}_{u,r}(\rd)\hookrightarrow \M(\rd)\quad \text{for all $r$ with $p\leq r\leq u$}.
\end{equation}

  \end{remark}

 \begin{remark}\label{emb-Linf}
   Let us recall Sawano's observation, cf. \cite{Saw18},  that {if} $\ \inf_{t>0} \varphi(t)>0$, then $\mathcal{M}_{\varphi,p}(\rd)\hookrightarrow L_\infty(\rd)$, while 
$\ \sup_{t>0} \varphi(t)<\infty\ $ implies $ L_\infty(\rd)\hookrightarrow\mathcal{M}_{\varphi,p}(\rd)$, leading as a special case to \eqref{M1p}.
 In particular, when $\varphi\in\Gp$, $0<p<\infty$, then according to \cite[Thms. 27, 28]{FHS-MS-2},
     \begin{equation}\label{emb-Linf-sharp}
       \M(\rd) \hookrightarrow L_\infty(\rd) \quad\text{if, and only if,}\quad 
       \inf_{t>0} \varphi(t)>0,
  \end{equation}
and
     \begin{equation}\label{Linf-emb-sharp}
       L_\infty(\rd) \hookrightarrow \M(\rd) \quad\text{if, and only if,}\quad 
       \sup_{t>0} \varphi(t)<\infty.
  \end{equation}
  \end{remark}

We consider the following examples.

\begin{examples}\label{exm-Gp}
\begin{enumerate}[\bfseries (i)]
\item The function 
\begin{equation}\label{example1} \varphi_{u,v}(t)=
\begin{cases}
t^{\nd/u} & \text{if}\qquad t\le 1,\\ 
t^{\nd/v} & \text{if}\qquad t >1, 
\end{cases}
\end{equation}
with $0<u,v<\infty$ belongs to $\Gp$ with $p=\min(u,v)$. In particular, taking $u=v$, the
function $\varphi(t)=t^{\frac{\nd}{u}}$ belongs to $\mathcal{G}_p$ whenever  $0<p\leq u<\infty$. 
\item The function  $\varphi(t)=\max(t^{\nd/v},1)$ belongs to $\mathcal{G}_v$. It corresponds to \eqref{example1} with $u=\infty$.
\item The function  $\varphi(t)=\sup\{s^{\nd/u}\chi_{(0,1)}(s): s\le t\}= \min(t^{\nd/u},1)$ defines an equivalent (quasi)-norm in $\mathcal{L}^{\sigma}_p(\rd)$, $\sigma= -\frac{\nd}{u}$, $p\le u$. The function $\varphi$  belongs to $\mathcal{G}_u\subset \mathcal{G}_p$. It corresponds to \eqref{example1} with $v=\infty$. 
\item  The   function $\varphi(t)=t^{\nd/u}(\log (L+t))^a$, with $L$ being a sufficiently large constant, belongs to $\mathcal{G}_u$ if  $0< u<\infty$ and $a\leq 0$. 
\end{enumerate}
\smallskip
Other examples   can be found e.g. in \cite[Ex.~3.15]{Saw18}.
\end{examples}

If $f$ is an extended complex-valued measurable function on
$\rd$ which is finite a.e., then the decreasing
rearrangement of $f$ is the function defined on $[0,\infty)$ by
\begin{equation}  \label{f*}
f^{*}(t):= \inf \{\sigma > 0: \mu (f,\sigma)\leq t\},\quad t \geq
0,
\end{equation}
with $\mu (f,\sigma)$ being the distribution function given by
$$
\mu (f,\sigma):= \big|\{x\in\rd:|f(x)|>\sigma\}\big|,
\quad \sigma \geq 0.
$$
As usual, the convention $\inf \emptyset =\infty$ is assumed.

\begin{definition}
\label{defi_eg}
Let $X=X(\rd)$ be some quasi-normed function space on $\rd$. 
The {\em growth envelope function} $\egX : (0,\infty) \rightarrow [0,\infty]$ of $\ X\ $
is defined by
\[
\egX(t) := \sup_{\|f\mid X\|\leq
1} f^*(t), \qquad t\in (0,\infty).
\]
\end{definition}

This concept was first introduced and  studied in \cite[Chapter~2]{T-func} and
\cite{HaHabil}; see also \cite{Ha-crc}. For convenience, we recall some
properties. We obtain, strictly
speaking, equivalence classes of growth envelope functions
when working with equivalent quasi-norms in $ X $ as we shall usually
do. But we do not want to
distinguish between representative and equivalence class in what
follows and thus stick at the notation introduced above.

The following result comes from \cite[Prop.~3.4]{Ha-crc}.

\begin{proposition} \label{properties:EG}
\begin{enumerate}[{\bfseries\upshape (i)}]
\item
Let $X_i=X_i(\rd)$, $i\in\{1,2\}$, be some
function spaces on $\rd$. Then $X_1\hookrightarrow X_2$ implies
that there exists a positive constant $C$ such that, for all $t\in(0,\infty)$,
\beq
\egv{X_1}(t)\ \leq \ C\ \egv{X_2}(t).
\label{eg-XX}
\eeq
\item
The embedding $X(\rd)\hookrightarrow L_{\infty}(\rd)$ holds if,
and only if, $\egX$ is bounded.
\end{enumerate}
\end{proposition}

\begin{remark}
Let $X=X(\rd)$ be a rearrangement-invariant Banach function space, $t>0$, and $ A_t
\subset \rd $ with $|A_t| = t$, then the {fundamental function}
$\fundfn{X}$ of $X$ is defined by $\ 
\fundfn{X}(t) = \big\| \charfn{A_t} \big| X\big\| $. In
\cite[Sect.~2.3]{Ha-crc} we proved that in this case
\[\egX(t) \sim  \frac{1}{\fundfn{X}(t)}, \quad t>0.\]
\end{remark}

\begin{example}
Basic examples for spaces $X$ are the well-known Lorentz spaces $L_{p,q}(\rd)$; for
  definitions and further details we refer to \cite[Ch.~4, 
  Defs.~4.1]{BS}, for instance. This scale represents a natural refinement of the scale of Lebesgue spaces. The following result is proved in \cite[Sect. 2.2]{Ha-crc}. Let $0<p<\infty$, $0<q\leq \infty$. Then
\begin{equation}
\egv{L_{p,q}(\rd)}(t) \;\sim\; t^{-\frac1p} \; ,\qquad t >0.
\label{Lp-global}
\end{equation}
\label{Prop-Lpq}
\end{example}

\begin{remark}\label{R-uGindex}
We dealt in the papers and books mentioned above usually with the so-called {growth
envelope} $\envg(X)$ of some function space $X$: the envelope function $\egX$
is then equipped with some additional fine index $\uGindex$ that contains
further information.  More precisely, assume $X\not\hookrightarrow L_\infty$. For convenience, let $\egX$ be continuously differentiable. Then the number $\uGindex$, $0<\uGindex\leq\infty$, is defined as the infimum of all numbers $ v$, $ 0<v\leq \infty$, such that
\begin{equation}
\left(-\int\limits_0^{\varepsilon}\left(\frac{f^\ast(t)}{\egX(t)}\right)^v\; \frac{(\egX)'(t)}{\egX(t)}\ud t\right)^{1/v}\leq c\,\|f\vert X\| \label{uX}
\end{equation}
$($with the usual modification if $v=\infty)$ holds for some $c>0$
and all $f\in X$. The couple
\[
\envg(X) =\left(\egX(\cdot),\uGindex\right)
\]
is called {\em (local) growth envelope} for the function space $X$. We shall benefit from the following observation below: let $X_i$, $i=1,2$, be some function
spaces on $\rd$ with $X_1\hookrightarrow X_2$. Assume for their growth envelope functions
\[
\egv{X_1}(t)\ \sim \ \egv{X_2}(t),
\qquad t\in (0,\varepsilon),
\]
for some $\varepsilon>0$. This implies
\begin{equation}\label{uG-mon}
\uGv{X_1}\ \leq \  \uGv{X_2}
\end{equation}
for the corresponding indices, cf. \cite[Prop.~4.5]{Ha-crc}. Let us finally mention that the above Example~\ref{Prop-Lpq} then reads as
\begin{equation}
\envg(L_{p,q}(\rd)) = \left(t^{-\frac1p}, q\right),\quad \text{in particular,}\quad \envg(L_{p}(\rd)) = \left(t^{-\frac1p}, p\right),
\label{env-Lp}
\end{equation}
see \cite[Thm.~4.7]{Ha-crc}. 
\end{remark}

In view of the embeddings $L_u(\rd)\hookrightarrow\Mu(\rd)$, $p\leq u$, the monotonicity of the growth envelope function \eqref{eg-XX}, and the result for Lebesgue spaces \eqref{Lp-global} (with $p=q$) we immediately obtain
\[
\egv{\Mu(\rd)}(t) \geq \ C\ t^{-\frac1u},\quad t>0,
\]
whenever $0<p\leq u<\infty$. However, the final result obtained in \cite{HaSM3} in case of $p<u$ is surprisingly much stronger.

\begin{proposition}\label{eg-M} 
Let $0<p<u<\infty$. Then
$$
\egv{\Mu(\rd)}(t)=\infty, \quad t>0.
$$
\end{proposition}

\begin{remark}
  Note that both the construction in \cite{HaSM3} and also the result heavily depend on the underlying $\rd$ where the spaces are defined. In case of spaces on a bounded domain $\Omega\subset\rd$ we refer to our results in \cite{HMSS-morrey} and \cite{HSS-morrey}.
\end{remark}

We collect what is already known about growth envelopes of spaces $\M(\rd)$, recall also \eqref{lim-exist}.

\begin{corollary}\label{egM-coll}
  Let $0<p<\infty$ and $\varphi\in\Gp$ with $\varphi(1)=1$.
  \begin{enumerate}[{\bfseries\upshape(i)}]
\item $\egv{\M(\rd)}$ is bounded if, and only if, $ \lim_{t\rightarrow 0^+} \varphi(t)>0$.
\item
Assume that $\varphi(t) \sim t^{\nd/p},\; t>1$. Then there exists some $c>0$ such that for all $t>0$,
\[
\egv{\M(\rd)}(t) \leq \ c\ t^{-\frac1p}.
\]
\item
  Assume that there exists some $u> p$ such that $\varphi(t) \leq c t^{\nd/u}$, $t>0$. Then
  \[
\egv{\M(\rd)}(t) = \infty, \quad t>0.
\]
\end{enumerate}
\end{corollary}

\begin{proof}
  Part (i) is a combination of Proposition~\ref{properties:EG}(ii) and Remark~\ref{emb-Linf}, in particular, \eqref{emb-Linf-sharp}, 
  since $\inf_{t>0} \varphi(t)= \lim_{t\rightarrow 0^+} \varphi(t)$. Likewise \eqref{Mphi-Lr}, Proposition~\ref{properties:EG}(i) and \eqref{Lp-global} (with $p=q$) lead to (ii). The last result is a consequence of  \eqref{emb-Mu-Mphi} together with Propositions~\ref{properties:EG}(i)
and \ref{eg-M}. 
\end{proof}

Now we study the general case $0<p<\infty$ and $\varphi\in\Gp$ with $\varphi(1)=1$. Recall that by \eqref{lim-exist} the limits $
  \lim_{t\to\infty} \varphi(t) t^{-\frac{\nd}{p}} \in [0,1]$ and  $\lim_{t\to 0^+} \varphi(t) \in [0,1]$ always exist. In view of Corollary~\ref{egM-coll}(i) we may always assume $ \lim_{t\rightarrow 0^+} \varphi(t)=0$ now.

We begin with some preparatory lemma, which might also be of independent interest.

\begin{lemma}\label{clever-lemma-LS}
Let  $0<p<\infty$ and   $\varphi\in \mathcal{G}_p$ with $\varphi(1)=1$ and $ \lim_{t\rightarrow 0^+} \varphi(t)=0$. 
Let $A_1,\ldots , A_n$ be  measurable sets in $\rd$ of finite measure, which might be assumed to be pairwise disjoint.  Let $a_1>a_2>\ldots > a_n>0$ and  let 
\[ f(x)= \sum_{j=1}^n a_j\chi_{A_j}(x)\]
be a simple function such that $\| f|L_p(\rd)\|=1$. Then there exists a simple function $g$ on $\rd$ such that  $g^*(t)=f^*(t)$, $t>0$,  and  $\|g|\M(\rd)\|\le 1$. 
\end{lemma}

\begin{proof}
Let $t_j=\sup \{t>0: a_j \varphi(t) \le 1\}$, $j=1,\ldots, n$. {Please note that it may happen that $t_j=\infty$ or that $t_j<\infty$ but $a_j \varphi(t_j)>1$.}     
If $t_j {\ge} |A_j|^\frac{1}{\nd}$, then we take  { $m_j=2$}, otherwise let $m_j$ be the smallest integer such that {$(m_j-1)^\frac{1}{\nd}t_j\ge |A_j|^\frac{1}{\nd}$}. 
{ We put} 
\[
\tau_j= 	\left(\frac{|A_j|}{m_j}\right)^\frac{1}{\nd} , \quad j=1, \ldots, n. 
\]
{Thus $\tau_j < t_j $} which implies $\varphi(\tau_j) a_j \leq 1$. Let   $M_j=m_1+\ldots +m_j$,   for $j=1, \ldots, n$, $M_0=0$, { $A= 2+  \max\{|A_j|:j=1,\ldots, n\}$}, and let $x_k=(2^k A,0,\ldots, 0)$, $k\in \nat$. { We consider cubes $Q(x_k,\tau_j)$  centred at $x_k$ with edges parallel to coordinate axes and side length $\tau_j$, that is, with volume $|A_j|/m_j$.}   We define the following simple function 
\begin{align*}
	g(x) = \begin{cases}
            a_1 & \text{if}\qquad x\in E_1=\bigcup_{k=1}^{M_1}Q(x_k,\tau_1), \\
            a_2 & \text{if}\qquad x\in E_2=\bigcup_{k=M_1+1}^{M_2}Q(x_k,\tau_2), \\
             & \vdots \\
            a_m & \text{if}\qquad x\in E_m=\bigcup_{k=M_{m-1}+1}^{M_m}Q(x_k,\tau_m), \\
            0& \text{otherwise}.
           \end{cases}
\end{align*}
The cubes $Q(x_k,\tau_j)$ 
are pairwise disjoint, thus  the sets $E_j$ 
are also  pairwise disjoint; therefore  $g$ is a well-defined simple function. Note also that
$$
|E_j|=\sum_{k=M_{j-1}+1}^{M_j} |Q(x_k,\tau_j)|= m_j \tau_j^{\nd}=|A_j|, \quad j=1,\ldots, n,
$$
hence $g^*(t)=f^*(t)$ for any $t>0$.

Let $Q$ be a cube such that $|Q|\le 1$. {The distance between two different cubes $Q(x_k,\tau_j)$ is bigger {than} $2$ so 
there is at most 
one cube $Q(x_k,\tau_j)\subset E_j$ with non-empty intersection with $Q$. }

If $Q\cap  Q(x_k,\tau_j)\not= \emptyset$ and $|Q|\le |Q(x_k,\tau_j) |$, then by the monotonicity of $\varphi$,
\begin{align*}
	\varphi(\ell(Q))\left( \frac{1}{|Q|}\int_Q|g(x)|^p \dint x\right)^{1/p}\le \varphi(\ell(Q(x_k,\tau_j)))\left( \frac{1}{|Q|}\int_Q|g(x)|^p \dint x\right)^{1/p}
\le \varphi(\tau_j) a_j \le 1.
\end{align*} 

If $Q\cap  Q(x_k,\tau_j)\not= \emptyset$ and $|Q(x_k,\tau_j) |<|Q|\le 1$, then
\begin{align*}
	 \varphi(\ell(Q))\left( \frac{1}{|Q|}  \int_Q|g(x)|^p \dint x\right)^{1/p}  \le &   \frac{\varphi(\ell(Q))|Q|^{-1/p}}{\varphi(\ell(Q(x_k,\tau_j)))|Q(x_k,\tau_j))|^{-1/p}} \times  \\
    &	\times \varphi(\ell(Q(x_k,\tau_j)))\left( \frac{1}{|Q(x_k,\tau_j))|}\int_{Q\cap Q(x_k,\tau_j))}|g(x)|^p \dint x\right)^{1/p} \\
	 \le & \varphi(\tau_j) a_j \le 1 \nonumber
\end{align*}
since the function $\varphi(t)t^{-\nd/p}$ is decreasing. 

If $|Q|\ge 1$, then $\varphi(\ell(Q))|Q|^{-1/p} \le \varphi(1) = 1$, again due to the monotonicity of $\varphi(t)t^{-\nd/p}$, and hence
\begin{equation*}
	\varphi(\ell(Q))\left( \frac{1}{|Q|}  \int_Q|g(x)|^p \dint x\right)^{1/p}\le \left( \int_Q|g(x)|^p \dint x\right)^{1/p} \le \left(  \int_{\rd}|g(x)|^p \dint x\right)^{1/p} = \left(  \int_{\rd}|f(x)|^p \dint x\right)^{1/p} =1 .
\end{equation*}
Consequently, taking the supremum over all cubes $Q$, we obtain $\|g| \M(\rd)\|\leq 1$.
  \end{proof}


  Now we are ready to characterise completely the behaviour of the growth envelope functions for spaces $\M(\rd)$ in the remaining cases apart from Corollary~\ref{egM-coll}(i).
  
\begin{theorem}\label{egM-LS}
Let  $0<p<\infty$,  $\varphi\in \mathcal{G}_p$ with $\varphi(1)=1$ and $\lim_{t\rightarrow 0^+} \varphi(t)=0$. 
\begin{enumerate}[{\bfseries\upshape (a)}]  
\item	If $\lim_{t\rightarrow \infty} t^{-\nd/p}\varphi(t)=0$, then 
	\begin{equation}\label{00}
	\egv{\M(\rd)}(t)=\infty , \qquad t>0. 
	\end{equation}
\item
        If $\lim_{t\rightarrow \infty} t^{-\nd/p}\varphi(t)=c>0$, then 
	\begin{equation}\label{0c} 
		\egv{\M(\rd)}(t)\sim t^{-1/p} , \qquad t>0.
	 \end{equation}
\end{enumerate}
\end{theorem}

\begin{proof}
	\emph{Step 1.} First we assume that $\lim_{t\rightarrow \infty} t^{-\nd/p}\varphi(t)=0$. 
	Let
	 \[g_{R}(x) = \frac{1}{\varphi(R)}\chi_{B(0,R)}(x), \qquad 0< R \le 1 .\] 
	If $0<r\le R$, then for all $x\in\rd$,
	\[ \varphi(r)\left(\frac{1}{|B(x,r)|}\int_{B(x,r)} |g_R(y)|^p \dint y \right)^\frac{1}{p} \le \frac{\varphi(r)}{\varphi(R)}\le 1 .  \]
		If $R<r$, then for all $x\in\rd$,
	        \[ \varphi(r)\left(\frac{1}{|B(x,r)|}\int_{B(x,r)} |g_R(y)|^p \dint y \right)^\frac{1}{p} \le \frac{\varphi(r)}{\varphi(R)} \left(\frac{|B(x,r) \cap B(0,R)|}{|B(x,r)|}\right)^\frac{1}{p} \le 
 \frac{\varphi(r)}{\varphi(R)} \left(\frac{|B(x,R)|}{|B(x,r)|}\right)^\frac{1}{p}
                \leq 1 
	\]
by \eqref{Gp-def}. So 
	\[
	\|g_R|\mathcal{M}_{\varphi,p}(\rd)\|=  1 
	\]
	for any $0<R\le 1$. 
	
Since $\lim_{t\rightarrow \infty} t^{-\nd/p}\varphi(t)=0$,  we can choose  a strictly increasing  sequence $(n_k)_{k\in\nat}$, $n_k\in \mathbb{N}$, such that $n_k^{-\frac{\nd}{p}}\varphi(n_k)\le k^{-\frac1p}$, $k=2,3,...$, and  $n_1=1$. We take a sequence $(m_k)_{k\in\nat}$, $m_k\in \mathbb{N}_0$,  defined inductively by the following relation
\begin{align*}
	m_1 = &\; 0,\\
	m_2\ge & \;2(n_2+1), \\ 
	\vdots & \;\\
	m_k\ge & \; 2\max \big(m_{k-1}+n_2+1, m_{k-2}+n_3+1, \ldots,  m_2+n_{k-1}+1, n_k+1\big).
\end{align*}
Let
\[ f_R(x) =\sum_{k=1}^\infty g_R\big(x-(m_k,0,\ldots,0)\big) .\]
Then any ball $B(x,r)$, $0<r\le 1$,  intersects at most one support $\supp g_R\big(x-(m_k,0,\ldots,0)\big)$,  and any ball $B(x,r)$, $n_{k}\le r < n_{k+1}$, $k=1,2,3,\ldots$, intersects at most $k$ supports $\supp g_R\big(x-(m_k,0,\ldots,0)\big)$. In consequence,
\begin{eqnarray*}
		\|f_R|\mathcal{M}_{\varphi,p}(\rd)\|=  1 .
\end{eqnarray*} 
On the other hand it should be clear that  for any positive $\sigma$ with $0<\sigma < \varphi(R)^{-1}$,  we have 
\begin{equation*}
	\mu(f_R,\sigma)= |\{x\in \rd:\; |f_R(x)|>\sigma\}| = \sum_{k=1}^{\infty} |B(m_k,R)| = \infty .
\end{equation*} 
Now taking the supremum over $R$ we get 
\[\egv{\M(\rd)}(t)\ge \sup_{0<R\le 1} f^*_R(t)\ge \sup_{0<R\le 1} \frac{1}{\varphi (R)} =\infty , \qquad t>0,\]
since  $\lim_{t\rightarrow 0^+} \varphi(t)=0$.

\emph{Step 2.}
Now assume $\lim_{t\rightarrow \infty} t^{-\nd/p}\varphi(t)=c>0$. Then $\varphi(t) \sim t^{\nd/p},\; t>1$, and thus the upper estimate in \eqref{0c}  follows from Corollary \ref{egM-coll}(ii).
Now we deal with the lower estimate in  \eqref{0c} and apply Lemma~\ref{clever-lemma-LS}. Let $s>0$ and $ A_s\subset \rd $ with $|A_s|=s$. Consider the simple function
\begin{equation}
f_s= s^{-\frac1p} \charfn{A_s}.
\label{f_sr}
\end{equation}
Then 
\[
f_s^\ast(t) = s^{-\frac1p} \charfn{[0,s)}(t), \quad t>0, 
  \qquad\text{and}\qquad \left\| f_s|L_p(\rd)\right\| =1. 
  \]
Now, by Lemma~\ref{clever-lemma-LS}, for any $s>0$ there exists some $g_s\in \M(\rd)$ with $\|g_s | \M(\rd)\|\leq 1$ and $g_s^\ast(t)=f_s^\ast(t)$, $t>0$. Consequently,
\begin{align*}
  \egv{\M(\rd)}(t) \geq \sup_{s>0} g_s^\ast(t) =  \sup_{s>0} s^{-\frac1p} \charfn{[0,s)}(t) =  \sup_{s>t} s^{-\frac1p} = t^{-\frac1p}. 
\end{align*}
\end{proof}


\begin{corollary}\label{egM-finite}
  Let $0<p<\infty$ and $\varphi\in\Gp$ with $\varphi(1)=1$. Assume that 
  $ \lim_{t\rightarrow 0^+} \varphi(t)=0$. Then
  \begin{equation*}
    \egv{\M(\rd)}(t) = \infty \quad \text{for all}\quad t>0\qquad\text{if, and only if,}\qquad \lim_{t\rightarrow \infty} t^{-\nd/p}\varphi(t)=0. 
    \end{equation*}
\end{corollary}

\begin{proof}
  The sufficiency is already covered by Theorem~\ref{egM-LS}(a), while the necessity is a consequence of Theorem~\ref{egM-LS}(b) (or Corollary~\ref{egM-coll}(ii) in view of \eqref{Gp-tlarge}). Note that by \eqref{lim-exist} there are no further possibilities for $\varphi\in \Gp$.
\end{proof}

\begin{remark}\label{only-choices}
  Let us briefly summarise this complete characterisation of the growth envelope function for $\M(\rd)$ when $0<p<\infty$, $\varphi\in\Gp$ with $\varphi(1)=1$. In view of Corollary~\ref{egM-coll}(i), Theorem~\ref{egM-LS} and Corollary~\ref{egM-finite}, depending on the two limits in \eqref{lim-exist}, the situation can be completely described as follows,
  \[
  \begin{cases}
    \egv{\M(\rd)}(t) \ \sim \ 1 & \text{if, and only if,}\qquad \lim_{t\rightarrow 0^+} \varphi(t)>0, \\ 
\egv{\M(\rd)}(t) \ \sim \ t^{-1/p} \quad & \text{if, and only if,}\qquad \lim_{t\rightarrow 0^+} \varphi(t)=0 \quad\text{and}\quad \lim_{t\rightarrow \infty} t^{-\nd/p}\varphi(t)>0, \\
    \egv{\M(\rd)}(t) \ = \ \infty & \text{if, and only if,}\qquad \lim_{t\rightarrow 0^+} \varphi(t)=0 \quad\text{and}\quad \lim_{t\rightarrow \infty} t^{-\nd/p}\varphi(t)=0.
    \end{cases}
  \]
  So apart from the bounded (first) case, the generalised Morrey spaces $\M(\rd)$ -- in view of their unboundedness, measured in terms of growth envelope functions -- either behave like $L_p(\rd)$ or like, say, $\mathcal{M}_{u,p}(\rd)$, {for  $p<u$,} recall Proposition~\ref{eg-M}, our result from \cite{HaSM3}.
\end{remark}

We can even strengthen this observation as follows.

\begin{corollary}\label{envelope-intermediate}
  Let $0<p<\infty$ and $\varphi\in\Gp$ with $\varphi(1)=1$. Assume that 
  $ \lim_{t\rightarrow 0^+} \varphi(t)=0$ and $\lim_{t\rightarrow \infty} t^{-\nd/p}\varphi(t)>0$. Then
  \begin{equation*}
    \envg(\M(\rd)) = \left(t^{-\frac1p},p\right) = \envg(L_p(\rd)).
        \end{equation*}
\end{corollary}

\begin{proof}
  First we show that $\uGv{\M(\rd)} \leq p = \uGv{L_p(\rd)}$ and want to apply Remark~\ref{R-uGindex}, in particular, \eqref{uG-mon} for this. Plainly, by \eqref{Lp-global} and \eqref{0c},
  \[\egv{\M(\rd)}(t) \ \sim \ t^{-1/p} \ \sim\ \egv{L_p(\rd)}(t),\]
  while \eqref{Mphi-Lr} and our assumptions imply $\M(\rd) \hookrightarrow L_p(\rd)$. Thus \eqref{env-Lp} concludes the argument for $\uGv{\M(\rd)} \leq p$.

  Now we deal with the converse and have to prove that the existence of some $c>0$ such that 
  \begin{equation}\label{ugx-M}
\left(\int\limits_0^{\varepsilon} \left( t^{\frac1p} f^\ast(t)\right)^v \ \frac{\dint t}{t} \right)^{1/v}\leq c\,\|f\vert \M(\rd)\| 
\end{equation}
for all $f\in\M(\rd)$ implies $v\geq p$. We use an argument similar to \cite[Prop.~4.10]{Ha-crc} and Lemma~\ref{clever-lemma-LS} again. Let $v$ be such that \eqref{ugx-M} is satisfied for some $c>0$ and all $f\in\M(\rd)$.

Let $\ b=\{b_j\}_{j\in\nat}\ $ be a sequence of nonnegative
numbers, with $b\in\ell_p$ and $\| b|\ell_p\|>0$. For convenience we may assume $\ b_1 = \cdots = b_{J-1} = 0$, where $\ J\ $ is suitably chosen such that $\ 2^{-J} \sim \varepsilon$ given by \eqref{ugx-M}.
Let $\ x^0 \in\rd$, $\ \left|x^0\right|>4$, such that
$B(2^{-j}x^0, 2^{-j}) \ \cap \ B(2^{-k}x^0, 2^{-k})$
$= \emptyset\ $ for $ j\neq k$, $j,k\in\nat_0$. 
Let $m>J$. Then
 \[
  f_{b,m}(x) = \| b|\ell_p^m\|^{-1} \sum_{j=1}^m 2^{j\frac{\nd}{p}} \ b_j\ \charfn{B(2^{-j}x^0, 2^{-j})}(x), \quad x\in\rd, 
  \]
  belongs to $\ L_p(\rd)$, and due to the disjointness of the supports, $\| f_{b,m}| L_p(\rd)\|=1$. 

  On the other hand, $\ f_{b,m}^\ast\left(c\ 2^{-j\nd}\right) \ \geq \ c' \| b|\ell_p^m\|^{-1} \ b_j\
2^{j\frac{\nd}{p}}$, $\ j\leq m$.  Now Lemma~\ref{clever-lemma-LS} implies that for any such $b\in\ell_p$ and $m\in\nat$, there exists a function $f_{b,m}\in L_p(\rd)$ and hence a $g_{b,m}\in\M(\rd)$ with $\|g_{b,m}|\M(\rd)\|\leq 1$ and $g_{b,m}^\ast\left(c\ 2^{-j\nd}\right)  \sim f_{b,m}^\ast\left(c\ 2^{-j\nd}\right)  \geq \ c' \| b|\ell_p^m\|^{-1}  b_j\
2^{j\frac{\nd}{p}}$. 
Then by monotonicity arguments,
\begin{eqnarray*}
\left(\int\limits_0^\varepsilon \left( t^{\frac1p} \
    g_{b,m}^\ast(t)\right)^v \frac{\dint t}{t} \right)^{1/v} &\sim& 
    \left(\sum_{j=J}^m \left[ 2^{-j\frac{\nd}{p}} \
    g_{b,m}^\ast\left(c\ 2^{-j\nd}\right)\right]^v  \right)^{1/v}\\
&\geq & c'\ \| b|\ell_p^m\|^{-1} \left(\sum_{j=J}^m \left[2^{-j\frac{\nd}{p}} \  b_j\
   2^{j\frac{\nd}{p}} \right]^v  \right)^{1/v} \\
&\sim & \| b|\ell_p^m\|^{-1} \left(\sum_{j=J}^m b_j^v\right)^{1/v}, 
 \end{eqnarray*}
hence \eqref{ugx-M} applied to $g_{b,m}$ implies $\| b| \ell_v^m\| \leq c \|b|\ell_p^m\| \leq c \|b|\ell_p\|$ for arbitrary such sequences $b=\{b_j\}_{j\in\nat}$ of nonnegative numbers and arbitrary $m>J$. Letting $m\to\infty$, we thus arrive at $\| b| \ell_v\| \leq c \|b|\ell_p\|$, which  obviously requires $\ v\geq p$.
\end{proof}

  Recall the complete distinction of cases mentioned in Remark~\ref{only-choices} which only leaves the second as the most interesting case, that is, when
  \begin{equation}\label{phi-middle}
    0<p<\infty, \quad  \varphi\in\Gp, \quad \text{with}\quad  \varphi(1)=1, \quad \lim_{t\rightarrow 0^+} \varphi(t)=0, \quad \lim_{t\rightarrow \infty} t^{-\nd/p}\varphi(t)>0.
  \end{equation}
  Furthermore, this observation is strengthened by Corollary~\ref{envelope-intermediate} which characterises any such space $\M(\rd)$ as very similar to $L_p(\rd)$ -- in view of its unboundedness properties. So the question remains which functions $\varphi$ satisfy \eqref{phi-middle} apart from $\varphi(t) \sim t^{\nd/p}$ which leads to $\M(\rd)=L_p(\rd)$ and was already known. Clearly, the last assumption in \eqref{phi-middle} together with \eqref{Gp-tlarge} implies $\varphi(t) \sim t^{d/p}$ for $t\geq 1$, so the only possible modification concerns the local setting, i.e., $\varphi(t)$ for $0<t<1$.

\begin{example}
  We return to Example~\ref{exm-Gp}(i) and find that \eqref{example1} with $u>p$ is an appropriate example different from the $L_p$-setting, that is,
\[ \varphi_{u,p}(t)=
\begin{cases}
t^{\nd/u} & \text{if}\qquad t\le 1,\\ 
t^{\nd/p} & \text{if}\qquad t >1, 
\end{cases}
\]
with $0<p< u<\infty$. Theorem~\ref{egM-LS}(b) and Corollary~\ref{envelope-intermediate} lead in this case to
\[
\egv{\M(\rd)}(t) \sim t^{-\frac1p}, \quad t>0, \qquad\text{and}\quad 
  \envg(\M(\rd)) = \left(t^{-\frac1p},p\right)\quad \text{with}\quad \varphi=\varphi_{u,p}.
\]
To interpret this finding  we can thus conclude that we have some `mixture': {\em locally} we have the `classical' (proper) Morrey setting $\mathcal{M}_{u,p}(\rd)$, $u>p$, (which would correspond to $\varphi(t)\sim t^{\nd/u}$, $t>0$), while {\em globally} the situation resembles $L_p(\rd)$ (corresponding to $\varphi(t)\sim t^{\nd/p}$, $t>0$). Another possibility would be a local behaviour like in Example~\ref{exm-Gp}(iv) with $u=p$.
\end{example}

\section{Generalised  Besov-Morrey spaces}\label{gen-Nphi}

Let $\mathcal{S}(\rd)$ be the set of all Schwartz functions on $\rd$, endowed
with the usual topology,
and denote by $\mathcal{S}'(\rd)$ its topological dual, namely,
the space of all bounded linear functionals on $\mathcal{S}(\rd)$
endowed with the weak $\ast$-topology.
For all $f\in \mathcal{S}(\rd)$ or $f\in\mathcal{S}'(\rd)$, we
use $\mathcal{F} f $ to denote its Fourier transform, and $\mathcal{F}^{-1}f$ for its inverse.
Now let us define the generalised Besov-Morrey spaces introduced in \cite{NNS16}.

\begin{definition} \label{def-spaces} 
Let $0<p<\infty$, $0<q\leq \infty$, $s\in \rr$, and $\varphi\in \Gp$. 
Let $\eta_0,\eta\in \mathcal{S}(\rd)$ be nonnegative compactly supported functions satisfying
\begin{equation*}
\eta_0(x)>0 \quad \text{if}\quad x \in Q(2),
\end{equation*}
\begin{equation*}
0\notin \supp \eta\quad \text{and} \quad \eta(x)>0 \quad \text{if}\quad x \in Q(2) \setminus Q(1).
\end{equation*}
For $j\in\nat$, let $\eta_j(x):=\eta(2^{-j}x)$,  $x\in\rd$. 
The  generalised Besov-Morrey   space
  $\MB(\rd)$ is defined to be the set of all  $f\in\mathcal{S}'(\rd)$ such that
\begin{align*}
\big\|f\mid \MB(\rd)\big\|:=
\bigg(\sum_{j=0}^{\infty}2^{jsq}\big\| \mathcal{F}^{-1}(\eta_j  \mathcal{F} f)\mid
\M(\rd)\big\|^q \bigg)^{1/q} < \infty,
\end{align*}
with the usual modification made in case of $q=\infty$.
\end{definition}


\begin{remark}\label{rem-coinc}
  The above spaces have been introduced in \cite{NNS16}, also for applications we refer to the recent monograph \cite{FHS-MS-2}. 
  There the authors have proved that those spaces are independent of the choice of the functions $\eta_0$ and $\eta$ considered in the definition, as different choices lead to equivalent quasi-norms, cf. \cite[Thm.~1.4]{NNS16}.
 Let us mention that there exists a parallel approach to generalised Triebel-Lizorkin-Morrey spaces $\MF(\rd)$, but we shall not study these spaces here.
  When $\varphi(t):=t^{\frac{\nd}{u}}$ for $t>0$ and $0<p\leq u<\infty$, then 
\begin{align}\label{Nphi=Nu}
\MB(\rd)={\mathcal N}^s_{u,p,q}(\rd) 
\end{align}
are the usual Besov-Morrey spaces, which are studied in \cite{YSY10} and in the survey papers by Sickel \cite{s011,s011a}, see also the recent monograph \cite{FHS-MS-1}.  
Of course, we can recover the  classical Besov spaces $B^s_{p,q}(\rd)$ for any $0<p<\infty$, $0<q\leq \infty$, and $s\in \rr$, since 
\begin{align}\label{N=B}
B^s_{p,q}(\rd)={\mathcal N}^s_{p,p,q}(\rd).
\end{align}
Moreover, thanks to \eqref{emb-Linf-sharp} and \eqref{Linf-emb-sharp},   if $\varphi_0(t)\sim 1$, then for any $0<p<\infty$, $0<q\leq \infty$, and $s\in \rr$, it holds
\begin{equation}\label{B=Nphi}
B^s_{\infty,q}(\rd)=\mathcal{N}^s_{\varphi_0,p,q}(\rd).
\end{equation}
The elementary embeddings
\begin{equation} \label{elembb1}
\mathcal  N^{s+\varepsilon}_{\varphi,p,q_1} (\rd)\hookrightarrow
\mathcal  N^{s}_{\varphi,p,q_2}(\rd), \quad \varepsilon>0,
\end{equation}
and 
\begin{equation} \label{elembb2}
\mathcal  N^{s}_{\varphi,p,q_1} (\rd) \hookrightarrow \mathcal  N^{s}_{\varphi,p,q_2}(\rd) , \quad q_1\leq q_2,
\end{equation}
are natural extensions of their classical counterparts, cf. 
\cite[Prop.~3.3]{NNS16}.
\end{remark}

The following embedding assertion has been proved in \cite{hms22}.

\begin{theorem}  \label{main}
Let $s_i\in\rr$, $0<p_i<\infty$, $0<q_i\leq \infty$, and $\varphi_i\in {\mathcal G}_{p_i}$, for $i=1,2$. 
We assume without loss of generality that $\varphi_1(1)=\varphi_2(1)=1$. 
Let $ \varrho=\min(1,\frac{p_1}{p_2})$ and $\alpha_j= \sup_{\nu\le j}\frac{\varphi_2(2^{-\nu})}{\varphi_1(2^{-\nu})^\varrho}$, $j\in \mathbb{N}_0$. 

There is a continuous embedding 
\begin{equation} \label{embed1}
{\mathcal N}^{s_1}_{\varphi_1,p_1,q_1}(\rd) \hookrightarrow {\mathcal N}^{s_2}_{\varphi_2,p_2,q_2}(\rd)
\end{equation}
 if, and only if, 
\begin{align}\label{cond0}
\sup_{\nu\le 0}\frac{\varphi_2(2^{-\nu})}{\varphi_1(2^{-\nu})^\varrho} & < \infty , 
\intertext{and} 
\label{cond2}
\left\{   2^{j(s_2-s_1)} \alpha_j \frac{\varphi_1(2^{-j})^\varrho}{\varphi_1(2^{-j})}\right\}_{j\in \nat} & \in \ell_{q^*}  \qquad \text{where}  \quad \frac{1}{q^*}=\left(\frac{1}{q_2}-\frac{1}{q_1}\right)_+  .
\end{align}
\end{theorem}


Let $0<p<\infty$, $0<q\leq \infty$, $s\in \rr$, and $\varphi\in \Gp$. Then 
$$
\n(\rd)= \{\lambda=\{\lambda_{j,m}\}_{j,m}:  \| \lambda\mid \n\|<\infty \},
$$
where
$$
\| \lambda\mid \n\| =\Bigg(\sum_{j=0}^\infty 2^{jsq} 
\mathop{\sup_{\nu: \nu \leq j}}_{k\in \zd}\!  \varphi(2^{-\nu})^q \,2^{(\nu-j)\frac{\nd}{p} q}\Big(\!\! \mathop{\sum_{m\in\zd:}}_{Q_{j,m}\subset Q_{\nu,k}}\!\!|\lambda_{j,m}|^p\Big)^{\frac q p}\Bigg)^{1/q}
$$
with the usual modification if $q=\infty$; for details we refer to \cite{hms22}.\\

As usual,  regarding the study of growth envelopes of the spaces $\MB(\rd)$ of interest are the spaces which satisfy
$$\MB(\rd) \subset \Lloc(\rd) \qquad  \text{and} \qquad \MB(\rd) \not\hookrightarrow  L_{\infty}(\rd).$$
So we first clarify this situation.

Regarding embeddings into $ L_\infty(\rd)$, we obtain the following criterion. 

\begin{theorem}\label{emb-Linfty}
	Let $s\in\rr$, $0<p<\infty$, $0<q\leq \infty$,  and $\varphi\in \Gp$. Let $q'$ be given by $\frac{1}{q'}= (1-\frac{1}{q})_+$, as usual. Then
	\begin{equation}\label{embed-Linfty}
	\MB(\rd) \hookrightarrow L_\infty(\rd).
	\end{equation}
	if, and only if,  
	\begin{equation}\label{suff-Linf}
	\left\{ 2^{-js} \varphi(2^{-j})^{-1} \right\}_{j\in\no} \in \ell_{q'} .
	\end{equation}
	\end{theorem}
\begin{proof}
	The sufficiency follows from our previous result, cf.  \cite[Cor.~5.10]{hms22}. We are left to prove the necessity. 
	First we consider the case $0<q\leq 1$, that is, $q'=\infty$. Recall that by Remark~\ref{emb-Linf}, $\mathcal{M}_{\psi,p}(\rd)= L_\infty(\rd)$ when 
        $\psi(t)\equiv 1 $,  
        thus, by Definition~\ref{def-spaces}(i),  $\mathcal{N}^0_{\psi,p,\infty}(\rd)= B^0_{\infty,\infty}(\rd)$, see also \eqref{B=Nphi}. Since we assume \eqref{embed-Linfty}, then $\MB(\rd) \hookrightarrow L_\infty(\rd) \hookrightarrow B^0_{\infty,\infty}(\rd) =\mathcal{N}^0_{\psi,p,\infty}(\rd)$. Now an application of Theorem~\ref{main} implies that the condition $\left\{ 2^{-js} \varphi(2^{-j})^{-1} \right\}_{j\in\no} \in \ell_{\infty}$ is necessary for the embedding  \eqref{embed-Linfty}. 
	
	Now we assume that  $q>1$, that is, $q'<\infty$.  Let $\nu=\{\nu_j\}_{j\in\no}\in \ell_q$, $\|\nu|\ell_q\|\le 1$. We define a sequence $\gamma(\nu)=\{\gamma_j(\nu)\}_{j\in\no}$ by 
	\[\gamma_j(\nu) =  \nu_j 2^{-js} \varphi(2^{-j})^{-1}, \qquad j\in\no, \]
	and a sequence $\lambda(\nu)=\{\lambda_{j,m}(\nu)\}_{j,m}$, $j\in\no$, $m\in \zd$, by
	\[ \lambda_{j,m}(\nu)= 
	\begin{cases}
		\gamma_j(\nu) &\quad\text{if} \quad j\in \no\quad \text{and}\quad m=0,\\
		0 &\quad \text{otherwise} .
\end{cases}	
		 \]
	Then $\{\lambda_{j,m}(\nu)\}_{j,m}\in n^s_{\varphi,p,q}$ and $\|\{\lambda_{j,m}(\nu)\}_{j,m}|  n^s_{\varphi,p,q}\| \sim \|\nu| \ell_{q}\| $.   
	We use the wavelet decomposition of the space $\MB(\rd)$ constructed in \cite{hms22}. We put
	\[ f_\nu=\sum_{j=0}^\infty \gamma_j(\nu)\psi^G_{j,0}, \]
     for some fixed multi-index $G$. We refer to \cite{hms22} for the notation related to  wavelet systems.
     Then 
     \[ f_\nu\in \MB(\rd)\qquad \text{and}\qquad  \|f_\nu|\MB(\rd)\|\sim \|\{\lambda_{j,m}(\nu)\}_{j,m}|  n^s_{\varphi,p,q}\| \sim  \|\nu| \ell_{q}\| .\] 
Recall that we assume \eqref{embed-Linfty} to hold now. Then
\[  \bigg|  \sum_{j=0}^\infty \gamma_j(\nu) \bigg| \,\le \, c\ \|f_\nu |L_\infty(\rd)\|\le c'\ \| \nu | \ell_{q}\|\le c' .\]
In consequence, since this holds for all $\nu\in\ell_{q}$ with $\|\nu|\ell_q\|\leq 1$, 
\begin{align*}
		\left\|\left\{ 2^{-js} \varphi(2^{-j})^{-1} \right\}_{j\in \no }|\ell_{q'}\right\| = &\sup\left\{ \bigg|  \sum_{j=0}^\infty \nu_j 2^{-js} \varphi(2^{-j})^{-1}\bigg| \, :\, \|\nu|\ell_q\|\le 1 \right\} \\
	=	&  
		\sup\left\{ \bigg| \sum_{j=0}^\infty \gamma_j(\nu) \bigg|\, :\, \|\nu|\ell_q\|\le 1 \right\}\le C<\infty.		
\end{align*}
This concludes the argument.
\end{proof}

\begin{corollary}\label{m_b_emb}
  Let $s_i\in\rr$, $0<p_1<\infty$, $0<p_2\le\infty$, $0<q\leq \infty$ and $\varphi\in {\mathcal G}_{p_1}$. Denote again $\frac{1}{q^*}=(\frac{1}{q_2}-\frac{1}{q_1})_+$. 
\begin{enumerate}[{\bfseries\upshape (i)}] 
\item Assume $p_2<\infty$. Then
	\begin{equation}\label{emb-phi-infty}
	\mathcal{N}^{s_1}_{\varphi,p_1,q_1}(\rd) \hookrightarrow B^{s_2}_{p_2,q_2}(\rd)
	\end{equation}
	if, and only if,
	\begin{align}\label{nn-5}
p_1\leq p_2, \quad \varphi (t) \sim t^{\frac{\nd}{p_1}}\quad\text{if }\ t\geq 1, \quad \text{and} \quad
		\left\{2^{j(s_2-s_1) }\varphi(2^{-j})^{\frac{p_1}{p_2}-1}\right\}_j \in \ell_{q^\ast} .
	\end{align}
        \item Assume $p_2=\infty$. Then
	\begin{equation*}
	\mathcal{N}^{s_1}_{\varphi,p_1,q_1}(\rd) \hookrightarrow B^{s_2}_{\infty,q_2}(\rd) = \mathcal{N}^{s_2}_{\varphi_0, p_1, q_2}(\rd)
	\end{equation*}
	if, and only if,
	\begin{align}\label{nn-5a}
	\left\{2^{j(s_2-s_1) }\varphi(2^{-j})^{-1}\right\}_j \in \ell_{q^\ast} .
	\end{align}
        In particular, if $\lim_{t\to 0^+}\varphi(t)>0$, then
        \begin{align}\label{nn-5b}
          \mathcal{N}^{s_1}_{\varphi, p_1,q_1}(\rd)\hookrightarrow B^{s_2}_{\infty, q_2}(\rd)\quad \text{if, and only if,}\quad \begin{cases} s_1>s_2, & \text{or}\\ s_1=s_2 & \text{and}\quad q_1\leq q_2. \end{cases}
          \end{align}
        \end{enumerate}
      \end{corollary}

\begin{proof}
  This follows immediately from Theorem~\ref{main} together with \eqref{Nphi=Nu}, \eqref{N=B} when $p_2<\infty$, and together with \eqref{B=Nphi} when $p_2=\infty$.
\end{proof}


\begin{example}
  For later use let us return to the function $\varphi_{u,v}\in\Gp$ in Example~\ref{exm-Gp}(i), that is, \eqref{example1}, now with $v=p<u$. Then Theorem~\ref{main} applied to $\varphi_1(t)=t^{\nd/p}$, $\varphi_2(t)=\varphi_{u,p}(t)$, $p_1=p_2$, $q_1=q_2$ and $s_2=s$, $s_1=s+\frac{\nd}{p}-\frac{\nd}{u}$ leads to
  \begin{equation}\label{ex-last-1}
B^{s+\frac{\nd}{p}-\frac{\nd}{u}}_{p,q}(\rd) = \mathcal{N}^{s+\frac{\nd}{p}-\frac{\nd}{u}}_{\varphi_1, p,q}(\rd) \hookrightarrow \mathcal{N}^s_{\varphi_{u,p},p,q}(\rd),
  \end{equation}
  in particular, with $s=\frac{\nd}{u}$,
  \begin{equation}
    B^{\frac{\nd}{p}}_{p,q}(\rd) \hookrightarrow \mathcal{N}^{\frac{\nd}{u}}_{\varphi_{u,p},p,q}(\rd).\label{ex-last-2}
  \end{equation}
  Note that the same argument can be applied whenever $\varphi\in\Gp$ is such that $\varphi(t) \sim t^{\frac{\nd}{p}}$, $t\geq 1$, and $\{ 2^{js} \varphi(2^{-j})\}_{j\in\nat} \in \ell_\infty$. Then
\begin{equation}
    B^{\frac{\nd}{p}}_{p,q}(\rd) \hookrightarrow \MB(\rd).\label{ex-last-3}
  \end{equation}
Plainly \eqref{ex-last-3} corresponds to \eqref{ex-last-2} with $\varphi=\varphi_{u,p}$ and $s=\frac{\nd}{u}$.
\end{example}

Recall that the definition of  the growth envelope function makes sense only for spaces consisting of functions or regular distributions.  So we shall work with such smoothness spaces of Morrey type in the sequel. But we begin with a somewhat contrary consideration and check which spaces of that type contain one of the most prominent examples of a singular distribution, namely the Dirac $\delta$ distribution. In \cite[Exm.~3.2]{hms16} we found that
\begin{equation}\label{delta-N}
\delta\in \mathcal{N}^s_{u,p,q}(\rd) \quad\text{if, and only if,}\quad \begin{cases} s<\nd\left(\frac1u-1\right), & \qquad\text{or}\\ s= \nd\left(\frac1u-1\right) &\text{and}\quad q=\infty,\end{cases}
\end{equation}
which naturally extends the classical result for Besov spaces,
\begin{equation}\label{delta-B}
\delta\in\B(\rd) \quad\text{if, and only if,}\quad \begin{cases} s<\nd\big(\frac1p-1\big), & \qquad\text{or}\\ s= \nd\big(\frac1p-1\big) &\text{and}\quad q=\infty,\end{cases}
\end{equation}
cf. also \cite[page 34]{RS}. Our finding in the more general setting $\MB(\rd)$ is the following.

\begin{lemma}\label{delta-gen-N}
Let $s\in\rr$, $0<p<\infty$, $0<q\leq \infty$, and $\varphi\in \Gp$. Then
\begin{align}\label{delta1}
	\delta\in\MB(\rd) \quad\text{if, and only if,}\quad  \left\{2^{j(s+\nd)}  \varphi(2^{-j})\right\}_{j\in\no}\in\ell_q . 
	\end{align}
\end{lemma}

\begin{proof}
To prove the above equivalence we can use once more the wavelet decomposition of the space $\MB(\rd)$ constructed in \cite{hms22}. Wavelet coefficients $(\delta, \psi^G_{j,m})$ are different from zero if the origin belongs to the interior of the support of  $\psi^G_{j,m}$. By the dilation structure of the wavelet system {there is} a number $N$ independent of $j$ such that the above property holds exactly for $N$ functions $\psi^G_{j,m}$. In consequence we can find two constants $c_1,c_2>0$ such that 
\begin{equation}\nonumber
c_1 2^{j\nd} \le \Big(\sum_{\stackrel{m\in \mathbb{Z}^d}{Q_{j,m}\subset Q_{\nu,k}}} |2^{j\nd/2} (\delta, \psi^G_{j,m})|^p\Big)^{1/p} \le c_2 2^{j\nd}
\end{equation}	   
if the cube $Q_{\nu,k}$ contains at least one cube $Q_{j,m}$ related to the element of the wavelet system  $\psi^G_{j,m} $ with the above property. Otherwise the sum is zero. The constant $c_1,c_2$ can be chosen independent of $\nu$ and $k$. 
If $\nu< j$, then $\varphi(2^{-\nu})2^{(\nu-j)\frac{\nd}{p}} \le \varphi(2^{-j})$. 
So 
\begin{equation}\nonumber
	\sup_{\stackrel{\nu:\nu\le j}{k\in \zd}} \varphi(2^{-\nu})2^{(\nu-j)\frac{\nd}{p}} \Big(\sum_{\stackrel{m\in \zd}{Q_{j,m}\subset Q_{\nu,k}}} |2^{j\nd/2} (\delta, \psi^G_{j,m})|^p\Big)^{1/p}  \sim 2^{j\nd} \varphi(2^{-j}).
\end{equation} 
    Now \eqref{delta1} follows from the wavelet characterisation of the spaces, cf. \cite{hms22}.  
\end{proof}

\begin{remark}	The case  $\varphi(t) \sim t^{\nd/u}$, $u\geq p$, was proved in  \cite{hms16}, that is, \eqref{delta1} corresponds in this case to  \eqref{delta-N}.
\end{remark}

To study growth envelopes of smoothness spaces of Morrey type we need to assure that we deal with spaces of regular distributions.  For that reason we will be concerned now with this question for the spaces $\MB(\rd)$.  
  For convenience we briefly recall what is known in the case $\MB(\rd)$ with $\varphi(t) \sim t^{\nd/u}$, $0<p\leq u<\infty$. 
  If $u=p$, then $\MB(\rd)=\B(\rd)$ and the result can  be found in \cite[Thm.~3.3.2]{ST},  when $u>p$, then $\MB(\rd)=\mathcal{N}^s_{u,p,q}(\rd)$ and we refer to \cite[Thms.~3.3, 3.4]{hms16}, complemented by \cite[Thm.~3.4]{hms20}.  Let us denote by
\begin{align}
  \sigma_p^u := \nd \frac{p}{u} \left(\frac1p-1\right)_+,\label{sigma_p}
\end{align} 
adapting the notation in \cite{ht2021}. Then 

\begin{align}
\mathcal{N}^s_{u,p,q}(\rd) \subset \Lloc(\rd)  \quad\text{if, and only if,}\quad    \begin{cases} s>\sigma_p^u\ , & \text{or} \\
s=\sigma_p^u& \text{and} \ 0<q\leq
\min(\max(p,1),2).\end{cases}
\label{N-Lloc}
\end{align}

The critical smoothness $\sigma_p^u$ depends on $u\geq p$ in the setting $\varphi(t)\sim t^{\nd/u}$, so for general $\varphi\in \Gp$ we need to find a proper replacement of \eqref{sigma_p}. It turns out that again the limits \eqref{lim-exist} are important, in particular, the local behaviour of $\varphi(t)$ when $t\to 0^+$. In case of $\lim_{t\to 0^+} \varphi(t)>0$, we are locally in the situation \eqref{B=Nphi}, and \eqref{N-Lloc} suggests a condition like $s>0$ or $s=0$ and $0<q\leq 2$ (classical case). But what happens when $\lim_{t\to 0^+} \varphi(t)=0$ ? We found the following.

\begin{theorem}\label{emb-L1loc}
	Let $s\in\rr$, $0<p<\infty$, $0<q\leq \infty$, and $\varphi\in \Gp$ with  $\varphi(1)=1$.  
	\begin{enumerate}[{\bfseries\upshape (i)}] 
		\item  Let  $\lim_{t\rightarrow  0^+} \varphi(t) =c>0$.  Then
	\begin{equation}\label{embed-Lloc1}
		\MB(\rd) \subset \Lloc(\rd)
	\end{equation}
	if, and only if,  
	\begin{equation}\label{suff-Lloc2} 
		\begin{cases} s>0, & \quad \text{or}\\ 
		s=0 & \quad \text{and}\qquad 0<q\le 2. \end{cases} 
	\end{equation}  
  \item If  $\lim_{t\rightarrow  0^+} \varphi(t) =0$ and $p \ge 1$, then
  \eqref{embed-Lloc1}  { holds if, and only if,
  	\begin{equation}\label{suff2-Lloc2} 
  	\begin{cases} s>0, & \quad \text{or}\\ 
  		s=0 & \quad \text{and}\qquad 0<q\le \min (p,2). 
  	\end{cases} 
  	\end{equation}
}
  \item If  $\lim_{t\rightarrow  0^+} \varphi(t) =0$ and $0<p < 1$, then
  \eqref{embed-Lloc1} holds if
  	\begin{equation}\label{suff-cond-c}
  		\left\{ 2^{-js} \varphi(2^{-j})^{p-1} \right\}_{j\in\no} \in \ell_{q'}, 
  \end{equation}
where  $q'$ is given by $\frac{1}{q'}= (1-\frac{1}{q})_+$, as usual.
\end{enumerate}
\end{theorem}

\begin{proof}
\emph{Step 1.}~ For any $\varphi\in  \mathcal{G}_p$ with $\varphi(1)=1$ and $t>0$ 
we put
\begin{equation}\label{phi-local}
  \widetilde{\varphi}(t) = \min (\varphi(t),1) = \begin{cases} \varphi(t), & 0<t\leq 1, \\ 1, & t\geq 1.\end{cases}
\end{equation}
Then $\widetilde{\varphi}\in \mathcal{G}_p$ and $\widetilde{\varphi}(1)=1$,  which corresponds to $\varphi_{p,\infty}$ in \eqref{example1} in Example~\ref{exm-Gp}.
So we can consider the space $\mathcal{N}^s_{\widetilde{\varphi},p,q}(\rd)$.   
We first prove that 
\begin{equation}\label{phi-local-2}
	\MB(\rd) \subset \Lloc(\rd) \qquad \text{if, and only if,} \qquad 	\mathcal{N}^s_{\widetilde{\varphi},p,q}(\rd) \subset \Lloc(\rd),
\end{equation}  
i.e., only the local behaviour of $\varphi(t)$ for $t\to 0^+$ is essential for the characterisation of \eqref{embed-Lloc1}.  
It follows easily from Theorem~\ref{main} that $\MB(\rd) \hookrightarrow \mathcal{N}^s_{\widetilde{\varphi},p,q}(\rd) $. So it remains to show that $\MB(\rd) \subset \Lloc(\rd) $ implies $\mathcal{N}^s_{\widetilde{\varphi},p,q}(\rd) \subset \Lloc(\rd)$.  

Let $f\in  \mathcal{N}^s_{\widetilde{\varphi},p,q}(\rd)$. Then
the distribution  $f$ is regular, i.e., $f\in\Lloc(\rd)$, if for any cube $Q_{0,\ell}$, 
\begin{equation}\label{loc1}
\int_{Q_{0,\ell}} |f(x)|  \dint x  <\infty.
\end{equation} 
Let 
\begin{align} 
& f =\sum_{j=0}^\infty  {\sum_{m\in \zd}} \lambda_{j,m}a_{j,m}\label{loc2}\\ 
\text{with}\qquad  & \Bigg(\sum_{j=0}^\infty 2^{jsq} 
\mathop{\sup_{\nu: \nu \leq j}}_{k\in \zd}\!  \widetilde{\varphi}(2^{-\nu})^q \,2^{(\nu-j)\frac{\nd}{p} q}\Big(\!\! \mathop{\sum_{m\in\zd:}}_{Q_{j,m}\subset Q_{\nu,k}}\!\!|\lambda_{j,m}|^p\Big)^{\frac q p}\Bigg)^{1/q}<\infty\label{loc3} 
\end{align}
be an atomic decomposition of $f$, cf.   
 {\cite{NNS16}},  and let 
\[M_{j,\ell}= \{m\in \zd:\; \supp a_{j,m}\cap Q_{0,\ell}\not= \emptyset\}.
\] 
Then there exists some $\nu_0\in \zz$ independent of $f$, and  a dyadic cube  $Q_{\nu_0,k_0}$, such that   
\[ \bigcup_{j,m\in M_{j,\ell}} \supp a_{j,m} \subset  Q_{\nu_0,k_0}.\] 
Moreover one can easily verify that 
\eqref{loc1} holds if 
\[
\int_{\rd} \left|g(x)\right| \dint x <\infty \, , \qquad \text{where}\qquad
 g= \sum_{j=0}^\infty\sum_{m\in M_{j,\ell}}\lambda_{j,m}a_{j,m}.
\] 
But 
\begin{align*}
	\mathop{\sup_{\nu: \nu \leq j}}_{k\in \zd}\!  \varphi(2^{-\nu}) \,2^{(\nu-j)\frac{\nd}{p}}\Big(\!\! \mathop{\sum_{m\in M_{j,\ell}:}}_{Q_{j,m}\subset Q_{\nu,k}}\!\!|\lambda_{j,m}|^p 	\Big)^{\frac 1 p} \; \le \; \varphi(2^{-\nu_0})\mathop{\sup_{\nu: \nu \le j}}_{k\in \zd}\!  \widetilde{\varphi}(2^{-\nu}) \,2^{(\nu-j)\frac{\nd}{p}}\Big(\!\! \mathop{\sum_{m\in\zd:}}_{Q_{j,m}\subset Q_{\nu,k}}\!\!|\lambda_{j,m}|^p 	\Big)^{\frac 1 p}.
\end{align*}
So \eqref{loc3} implies  
\[ \Bigg(\sum_{j=0}^\infty 2^{jsq} 
\mathop{\sup_{\nu: \nu \leq j}}_{k\in \zd}\!  \varphi(2^{-\nu})^q \,2^{(\nu-j)\frac{\nd}{p} q}\Big(\!\! \mathop{\sum_{m\in M_{j,\ell}:}}_{Q_{j,m}\subset Q_{\nu,k}}\!\!|\lambda_{j,m}|^p\Big)^{\frac q p}\Bigg)^{1/q}<\infty, 
\]
and thus $g\in \MB(\rd)\subset \Lloc(\rd)$. So in view of \eqref{phi-local} and \eqref{phi-local-2} we may assume now in the sequel that $\varphi\in\Gp$ with $\varphi(t)\sim 1$, $t\geq 1$. \\

\emph{Step 2.}~  
{Let $\lim_{t\rightarrow  0^+}\varphi(t)=c>0$.   Since now the function  $\widetilde{\varphi}(t) \sim 1$ defines $L_\infty(\rd)$, i.e., ${\mathcal M}_{{\widetilde{\varphi}},p}(\rd) = L_\infty(\rd)$ for any $p$, recall \eqref{emb-Linf-sharp} and \eqref{Linf-emb-sharp},  we have  
	\begin{equation*}
		B^s_{\infty,q}(\rd)=\mathcal{N}^s_{\widetilde{\varphi},p,q}(\rd),
	\end{equation*}
cf. \eqref{B=Nphi}. 
Thus it follows from Step 1 that the conditions for local integrability are the same as for the classical spaces $B^s_{\infty,q}(\rd)$.\\ }

\emph{Step 3.}~ Now  we assume that $\lim_{t\rightarrow  0^+}\varphi(t)=0$. We first deal with (ii) { and consider $p\geq 1$}. Let $\varphi_p(t) = t^{\nd/p}$ 
and $\widetilde{\varphi_p}$ be defined as in Step 1. Then 
\[\sup_{\nu\le 0}\frac{\widetilde{\varphi_p}(2^{-\nu})}{\varphi(2^{-\nu})}=1\quad \text{and}\quad \sup_{\nu\le j}\frac{\widetilde{\varphi_p}(2^{-\nu})}{\varphi(2^{-\nu})}= \sup_{0\le \nu\le j}\frac{\widetilde{\varphi_p}(2^{-\nu})}{\varphi(2^{-\nu})}= \sup_{0\le \nu\le j}\frac{1}{2^{\nu \nd/p}\varphi(2^{-\nu})} = 1,  \]
{since  $\varphi \in \Gp$}. 
So Theorem~\ref{main} gives us
\begin{equation*}
\MB(\rd)\hookrightarrow \mathcal{N}^s_{\widetilde{\varphi_p},p,q}(\rd).  
\end{equation*}   
But $\mathcal{N}^s_{{\varphi_p},p,q}(\rd)= B^s_{p,q}(\rd)$, recall \eqref{Nphi=Nu} and \eqref{N=B}, therefore it follows from Step 1 and \eqref{N-Lloc}  that 
\[\MB(\rd) \subset  \Lloc(\rd) \]
if $s>0$ or $s=0$ and $0<q\le \min(p,2)$. 

{
Conversely, assume $\MB(\rd) \subset  \Lloc(\rd)$. Then, by Step 1, $\mathcal{N}^s_{\widetilde{\varphi},p,q}(\rd)  \subset  \Lloc(\rd)$.
But by Theorem~\ref{main} we get 
\[ B^s_{\infty,q}(\rd) = \mathcal{N}^s_{{\varphi_0},p,q}(\rd) \hookrightarrow  
\mathcal{N}^s_{\widetilde{\varphi},p,q}(\rd) \] 
where $\varphi_0\in \mathcal{G}_p$ such that $\varphi_0(t)\sim 1$ for any $t$.  So the condition {$s>0$ or $s=0$ and} $q\le 2$ follows from a similar statement as for classical Besov spaces $B^s_{\infty,q}(\rd)$. 
If $1<p<2$ { and when $s=0$}, then the necessity of the condition $q\le p$ can be proved in the same way as in Step 2 of the proof of Theorem~3.4 in \cite{hms20}. 
At the end for any $p\ge 1$ we have $\mathcal{N}^0_{\widetilde{\varphi},p,q}(\rd)  \subset \mathcal{N}^0_{\widetilde{\varphi},1,q}(\rd) $.  So if $\mathcal{N}^0_{\widetilde{\varphi},1,q}(\rd)  \subset  \Lloc(\rd)$, then also $\mathcal{N}^0_{\widetilde{\varphi},p,q}(\rd)  \subset  \Lloc(\rd)$ for any $1<p<2$. This implies  $q\le p$ for any $1<p<2$, so $q\le 1$ if $p=1$.\\ }
 
\emph{Step 4.}~ Now we deal with (iii) and assume $\lim_{t\rightarrow  0^+}\varphi(t)=0$ and $0<p<1$. Then Theorem~\ref{main} implies  
\begin{equation}
	\MB(\rd)\hookrightarrow \mathcal{N}^0_{\widetilde{\varphi},1,{1}}(\rd),  
\end{equation} 
 if 
 \[ 	\left\{ 2^{-js} \varphi(2^{-j})^{p-1} \right\}_{j\in\no} \in \ell_{q'}, \]  
 since $\varrho=p$ and 
 \[\sup_{\nu\le 0}\frac{\widetilde{\varphi}(2^{-\nu})}{\varphi(2^{-\nu})^p}=1\quad \text{and}\quad \sup_{\nu\le j}\frac{\widetilde{\varphi}(2^{-\nu})}{\varphi(2^{-\nu})^p}= \sup_{0\le \nu\le j}\frac{\varphi(2^{-\nu})}{\varphi(2^{-\nu})^p}= \sup_{0\le \nu\le j}\varphi(2^{-\nu})^{1-p} = 1 . \]
 So the statement follows from Step 3. 
\end{proof}


\begin{remark}
  Let us point out, that in case (iii) we only have a sufficient condition in \eqref{suff-cond-c}, while in the classical case, that is, when $\varphi(t) \sim t^{\nd/u}$, we have equivalence in \eqref{N-Lloc}, cf. \cite{hms16,hms20}. 
  Moreover, note that -- neglecting the just mentioned gap in the necessary condition in (iii) for the moment -- we obtain that $\mathcal{N}^s_{u,p,q}(\rd) \subset \Lloc(\rd)$ if, and only if, $\mathcal{N}^s_{\varphi_{u,p},p,q}(\rd) \subset \Lloc(\rd)$, if, and only if, \eqref{N-Lloc} is satisfied, where $\varphi_{u,p}$ is given by Example~\ref{exm-Gp}(i) with $u\geq p$. Hence this is really a local matter in terms of $\varphi$, since for $u>p$ we have   $\mathcal{N}^s_{\varphi_{u,p},p,q}(\rd) \subsetneq  \mathcal{N}^s_{u,p,q}(\rd)$, but $\varphi_{u,p}(t)=t^{\nd/u}$ for $0<t<1$.
\end{remark}


 Before discussing the growth envelopes of spaces of regular distributions of type  $\MB(\rd)$, we recall what is known in the particular case of  $\varphi(t) \sim t^{\nd/u}$, $0<p\leq u<\infty$,
 $0<q\leq\infty$, and 
 $$
{\sigma^u_p }<s<\frac{\nd}{u} \quad \text{(sub-critical case)} \quad \text{or} \quad s=\frac{\nd}{u} \quad \text{and} \quad  q>1 \quad \text{(critical case)},$$
{cf. notation  in \eqref{sigma_p}}. \\
 If $u=p$, then $\MB(\rd)=\B(\rd)$ and  it holds, for  $\sigma_p :=\nd  \left(\frac1p-1\right)_+ <s<\frac{\nd}{p}$, or $s=\sigma_p$ and {$0<q\leq \min\left(\max(p,1),2\right)$,}
 \begin{equation}\label{egv-B-subcritical}
 \egv{\B(\rd)}(t) \sim  t^{-\frac1p+\frac  s \nd}, \quad 0<t<1,
 \end{equation}
and for $s=\frac{\nd}{p}$ and $q>1$,
 \begin{equation}\label{egv-B-critical}
 \egv{\B(\rd)}(t) \sim  |\log t|^{\frac{1}{q'}}, \quad 0<t<\frac{1}{2},
 \end{equation}
cf. \cite[Thms. 8.1, 8.16]{{Ha-crc}} {and \cite[Props. 8.12, 8.14]{{Ha-crc}}}.
When $u>p$, then $\MB(\rd)=\mathcal{N}^s_{u,p,q}(\rd)$ and it has been proved in \cite[Thm.~4.11]{hms16}, that
\begin{equation}\label{egv-N}
 \egv{\mathcal{N}^s_{u,p,q}(\rd)}(t) =\infty, \quad t>0, 
 \end{equation}
whenever $s<\frac{\nd}{u}$ or $s=\frac{\nd}{u}$ and $q>1$.

Similarly to Corollary \ref{egM-coll} we have the following.

\begin{corollary}\label{Cor-SM-1}
Let $s\in\rr$, $0<p<\infty$, $0<q\leq \infty$, and $\varphi\in \Gp$ with $\varphi(1)=1$.  Assume $\MB(\rd)\subset \Lloc(\rd)$. 
  \begin{enumerate}[{\bfseries\upshape(i)}]
\item  $\egv{\MB(\rd)}$ is bounded if, and only if, $\left\{ 2^{-js} \varphi(2^{-j})^{-1} \right\}_{j\in\no} \in \ell_{q'}$, where  $q'$ is given by $\frac{1}{q'}= (1-\frac{1}{q})_+$, as usual.
\item If $\varphi(t) \sim t^{\nd/p},\; t>1$,  then there exists some $c>0$ such that 
\[
\egv{\MB(\rd)}(t) \leq \ c\ t^{-\frac1p+\frac  s \nd}, \quad 0<t<1, 
\]
if $\sigma_p<s<\frac{\nd}{p}$, or
\[
\egv{\MB(\rd)}(t) \leq \ c\ |\log t|^{\frac{1}{q'}}, \quad 0<t<\frac{1}{2}, 
\]
if $s=\frac{\nd}{p}$ and $1<q\leq \infty$.
\item Assume there exists some $u> p$ such that $\varphi(t) \leq c \, t^{\nd/u}$, $t>0$. Then, for $s<\frac{\nd}{u}$ or $s=\frac{\nd}{u}$ and $1<q\leq \infty$, it holds
  \[
\egv{\MB(\rd)}(t) = \infty, \quad t>0.
\]
\end{enumerate}
\end{corollary}

\begin{proof}
Part (i) is a consequence of Proposition \ref{properties:EG}(ii) and Theorem~\ref{emb-Linfty}. 
 If $\varphi(t) \sim t^{\nd/p},\; t>1$, then
$\MB(\rd)\hookrightarrow B^s_{p,q}(\rd)$, therefore Proposition \ref{properties:EG}(i),  
\eqref{egv-B-subcritical}  and \eqref{egv-B-critical} lead to (ii). Part (iii) follows from Proposition \ref{properties:EG}(i) and  \eqref{egv-N},  since the assumption on $\varphi$ implies
$ \mathcal{N}^s_{u,p,q}(\rd) \hookrightarrow \MB(\rd)$.
\end{proof}

{
Recall our Remark~\ref{only-choices} above where we found that for the growth envelope function of Morrey spaces $\M(\rd)$ only three cases are possible, depending on the behaviour of $\varphi(t)$ and $t^{-\nd/p} \varphi(t)$ for $t \to 0^+$ and $t\to\infty$, respectively. In principle we obtain a similar behaviour now, though more refined, as in addition we have to take the interplay between the smoothness $s$ and the function $\varphi$ into consideration. We shall deal with these possible cases  separately. The first, least interesting case, is already presented in Corollary~\ref{Cor-SM-1}(i). From now on we shall always assume that  
\begin{equation} \label{cond1}
\left\{ 2^{-js} \varphi(2^{-j})^{-1} \right\}_{j\in\no} \not\in \ell_{q'}\ ,
\end{equation}
where  $q'$ is given by $\frac{1}{q'}= (1-\frac{1}{q})_+$, as usual. Note that by \eqref{lim-exist} we always know that the corresponding limits exist, but we shall distinguish whether they are $0$ or positive. Moreover, we shall observe that now also the limit of $t^{-\nd/p} \varphi(t)$ for $t\to 0^+$ is important, that is, whether it is finite or not. For better understanding we split the observation and deal with those cases separately as follows:
\begin{itemize}
\item
 Proposition~\ref{P-case-1}: $\ds \lim_{t\rightarrow \infty} \varphi(t) t^{-\frac{\nd}{p}}=0$,
\item
Proposition~\ref{P-case-2}:   $\ds \lim_{t\rightarrow \infty} \varphi(t) t^{-\frac{\nd}{p}}>0$ 
 and $\ds \lim_{t\rightarrow 0^+} \varphi(t) t^{-\frac{\nd}{p}}<\infty $,
\item
Proposition~\ref{P-case-3}:  $\ds \lim_{t\rightarrow \infty} \varphi(t) t^{-\frac{\nd}{p}}>0$ 
  and $\ds \lim_{t\rightarrow 0^+} \varphi(t) t^{-\frac{\nd}{p}}=\infty $, and $s=0$,
\item
Proposition~\ref{P-case-4}:   $\ds \lim_{t\rightarrow \infty} \varphi(t) t^{-\frac{\nd}{p}}>0$ 
  and $\ds \lim_{t\rightarrow 0^+} \varphi(t) t^{-\frac{\nd}{p}}=\infty $, and $s>0$. 
\end{itemize}

At first we obtain the counterpart of \eqref{egv-N}.

 \begin{proposition}\label{P-case-1}
   	Let $0<p<\infty$, $0<q\leq \infty$,  $\varphi\in \Gp$ with $\varphi(1)=1$, and $s\in \rr$ so that {$\MB(\rd) \subset \Lloc(\rd)$}. 
	Assume    \eqref{cond1} to hold. If
        \begin{equation} \lim_{t\rightarrow \infty} \varphi(t) t^{-\frac{\nd}{p}}=0,\label{case-A}
        \end{equation}
        then 
\begin{equation}\label{ef-eM-0}
\egv{\MB(\rd)}(t)=\infty, \quad t>0.
\end{equation}
 \end{proposition}

 \begin{proof}
The proof is an adaptation of the proof of Theorem 4.11 in \cite{hms16}.  In view of \eqref{case-A} we can construct  a strictly increasing  sequence $(n_\ell)_{\ell\in\nat}$, $n_{\ell}\in \mathbb{N}_0$, such that 
\begin{equation}\label{2.06-2}
	\varphi(2^{n_{\ell}})2^{-\frac{\nd n_{\ell}}{p}} \le {\ell^{-\frac{1}{p}}}, 
	\quad \ell\in  \mathbb{N}, 
\end{equation}
where we can take  $n_1=0$. 
Using  the same type of construction as in \cite{hms16}, for each $\ell\in \nat_0$ consider a shrinking sequence of cubes $Q_{0,m_{0,\ell}}\supsetneq
Q_{1,m_{1,\ell}}\supsetneq Q_{2,m_{2,\ell}}\supsetneq \ldots$ for an appropriately chosen sequence $\{m_{j,\ell}\}_{j\in\no}\subset \zd$ now  with $m_{0,\ell}:=(2^{n_\ell},0,0,\cdots,0)$ and $x_0^{\ell} \in \bigcap_j Q_{j,m_{j,\ell}}$. 
{Let $\phi$ be a $C^{\infty}$  function satisfying 
\begin{equation*}\label{phi-1}
\phi(x)\geq C_1 \quad \text{if} \quad |x|\leq C_2, \quad \phi(x)=0 \quad  \text{if}  \quad |x|>1,
\end{equation*}
and
\begin{equation*}\label{phi-2}
 \int_{\rd} x^{\gamma} \phi(x) \dint x=0 \quad \text{whenever} \quad |\gamma| \leq L \quad \text{with} \quad \gamma\in \no^\nd,
\end{equation*}
for some  constants $C_1>0$ and $C_2\in (0,1)$, and for fixed  $L\in \nat_0$  with $L\geq  \max(-1,\whole{\sigma_p-s})$.
 } 
Then, for any $j,\ell\in \no$, the function
$$
a_{j, m_{j,\ell}}(x):= {\phi}\bigl(2^{j+2}(x-x_0^{\ell})\bigr), \quad x\in\rd,
$$
is  supported in $Q_{j,m_{j,\ell}}$, and gives rise to a family of atoms for  $\MB(\rd)$. 
It follows from the condition \eqref{cond1} that 
\begin{equation}\label{2.06-1}
	\sum_{j=0}^\infty 2^{-jsq'} \varphi(2^{-j})^{-q'}
	= \infty 
\end{equation}
(with the usual modification if $q'=\infty$). We use an adapted version of an inequality of Landau, see \cite[Thm.~161, p.~120]{hardy}. 
If $1<q<\infty$, then it follows from the  Landau theorem, that remains valid  in the quasi-Banach setting, cf. \cite{CH15}, that {there exists} a sequence $v=(v_j)_{j\in\nat_0}\in \ell_q$ such that 
\begin{equation}
	\sum_{j=0}^\infty v_j2^{-js}\varphi(2^{-j})^{-1}	= \infty . 
\end{equation}
We may assume that $\|v|\ell_q\|=1$  and that $v_{j}\ge 0$, since $2^{-js}\varphi(2^{-j})^{-1}>0$.    
If $q=\infty$, then the same identity holds for $v=(1,1,1,\ldots)$.   
 We define a sequence $(\lambda_{j})_{j\in\nat_0}$ by 
 \begin{equation*}
 	\lambda_j= v_j2^{-js}\varphi(2^{-j})^{-1}, \quad j\in \no,
 \end{equation*}
and  a distribution $f\in\mathcal{S}'(\rd)$ by its atomic
decomposition, 
$$ f:= \sum_{j=0}^\infty \sum_{\ell=0}^\infty \lambda_j a_{j,m_{j,\ell}} \qquad \text{i.e., with}\qquad
\lambda_{j, m}=\begin{cases} \lambda_j, & \text{if}\quad m=m_{j,\ell}, \ j,\ell\in\no,
	\\ 0, & \text{otherwise}.\end{cases} $$

If $|Q_{\nu,m}|\le 1$, then 
 \begin{align} \label{08.06-1}
	\varphi(2^{-\nu})\, 2^{(\nu-j){\frac{\nd}{p}}} \Big(\sum_{m:Q_{j,m}\subset Q_{\nu,k}}\!\!|\lambda_{j,m}|^p\Big)^{\frac 1 p}   \le \varphi(2^{-\nu})\, 2^{(\nu-j){\frac{\nd}{p}}} 
	2^{-js}\varphi(2^{-j})^{-1} v_j \le 2^{-js} v_j 
\end{align}
since $\varphi\in  \mathcal{G}_p$. 
If the volume of the cube $Q_{\nu,m}$  is bigger than $1$, then it can contain several cubes $Q_{0,m_{0,\ell}}$. 
Let us denote by $\mu_\nu$ the number of such cubes. Then $-\nu\ge n_{\mu_\nu}$.
This follows from the definition of the points   $m_{0,\ell}$. Hence,
  \begin{align} 
 \varphi(2^{-\nu})\, 2^{(\nu-j){\frac{\nd}{p}}} \Big(\sum_{m:Q_{j,m}\subset Q_{\nu,k}}\!\!|\lambda_{j,m}|^p\Big)^{\frac 1 p}  & \le \varphi(2^{-\nu})\, 2^{(\nu-j){\frac{\nd}{p}}} 
  2^{-js}\varphi(2^{-j})^{-1} \Big(\sum_{m:Q_{j,m}\subset Q_{0,m_{0,\ell}}\subset  Q_{\nu,k}}\!\!|v_{j}|^p\Big)^{\frac 1 p} \nonumber \\
     & \le \varphi(2^{{n_{\mu_\nu}}})\, 2^{{-}  \frac{\nd}{p}n_{\mu_\nu}} \mu_\nu^{1/p} 2^{-j\frac{\nd}{p}}
   \varphi(2^{-j})^{-1}  2^{-js} v_j   
     \le  2^{-js} v_j  \label{2.06-4} 
  \end{align}
where we used \eqref{2.06-2}  and recall that $\varphi\in \mathcal{G}_p$ and $\varphi(1)=1$. So we arrive at 
\begin{align*}
	\sup_{\nu: \nu \le j; k\in \zd}\! \varphi(2^{-\nu})\, 2^{(\nu-j){\frac{\nd}{p}}}
	\Big(\sum_{m:Q_{j,m}\subset Q_{\nu,k}}\!\!|\lambda_{j,m}|^p\Big)^{\frac 1 p}
	\leq 2^{-jsq} v_j,
\end{align*}   
cf. \eqref{08.06-1} and \eqref{2.06-4}. So 
$$\|f|\MB(\rd)\|\le c \big(\sum_{j=0}^\infty v_j^q\big)^\frac{1}{q}=c$$ 
{with the usual modification if $q=\infty$.}

The decreasing rearrangement of $f$ can be calculated as in the proof of \cite[Thm.~4.11]{hms16}, leading to $f^*(t)=\infty$ for any $t>0$.
   \end{proof}

 So we are left to deal with the case $\ \lim_{t\rightarrow \infty} \varphi(t) t^{-\frac{\nd}{p}}>0$ in the sequel. Thus Corollary~\ref{m_b_emb}(i) with $p_1=p_2$, $q_1=q_2$ implies
 \begin{equation}\label{MB-in-B}
\MB(\rd) \hookrightarrow \B(\rd).
 \end{equation}
 Now we state the counterpart of \eqref{egv-B-subcritical}, \eqref{egv-B-critical}.

 \begin{proposition}\label{P-case-2}
   	Let $0<p<\infty$, $0<q\leq \infty$,  $\varphi\in \Gp$ with $\varphi(1)=1$, and $s\in \rr$ so that {$\MB(\rd) \subset \Lloc(\rd)$}. 
	Assume    \eqref{cond1} to hold. If
\begin{equation}
  \lim_{t\rightarrow \infty} \varphi(t) t^{-\frac{\nd}{p}}>0\quad  \text{and}\quad \lim_{t\rightarrow 0^+} \varphi(t) t^{-\frac{\nd}{p}}<\infty,\label{case-B}
\end{equation}
then for $\nd  \left(\frac1p-1\right)_+ \leq s<\frac{\nd}{p}$, 
\begin{equation}\label{egv-N-subcritical}
 \egv{\MB(\rd)}(t) \sim  t^{-\frac1p+\frac  s \nd}, \quad 0<t<1,
 \end{equation}
and for $s=\frac{\nd}{p}$ and $q>1$,
 \begin{equation}\label{egv-N-critical}
 \egv{\MB(\rd)}(t) \sim  |\log t|^{\frac{1}{q'}}, \quad 0<t<\frac{1}{2}.
 \end{equation}
 \end{proposition}

 \begin{proof}
   Note that \eqref{case-B} implies  that
   $$ \varphi(t) \sim t^{\frac{\nd}{p}}, \quad t>0,
   $$
which by \eqref{Nphi=Nu} and \eqref{N=B} leads to $\MB(\rd) = \B(\rd)$. Thus \eqref{egv-B-subcritical}, \eqref{egv-B-critical} conclude the argument.
 \end{proof}

 Finally we are left to deal with the case when $\lim_{t\rightarrow \infty} \varphi(t) t^{-\frac{\nd}{p}}>0$ and $  \lim_{t\rightarrow 0^+} \varphi(t) t^{-\frac{\nd}{p}}=\infty$, and we shall split this case again in the settings when $s=0$ and $s>0$.

 \begin{proposition}\label{P-case-3}
   	Let $0<p<\infty$, $0<q\leq \infty$,  $\varphi\in \Gp$ with $\varphi(1)=1$, and $s=0$ so that {$\mathcal{N}^0_{\varphi,p,q}(\rd) \subset \Lloc(\rd)$}. 
	Assume    \eqref{cond1} to hold, that is, $\{\varphi(2^{-j})^{-1}\}_{j\in\no}\not\in\ell_{q'}$. Let 
\begin{equation}
  \lim_{t\rightarrow \infty} \varphi(t) t^{-\frac{\nd}{p}}>0\quad  \text{and}\quad \lim_{t\rightarrow 0^+} \varphi(t) t^{-\frac{\nd}{p}}=\infty.\label{case-C}
\end{equation}
\begin{enumerate}[{\bfseries\upshape(i)}]
\item
  If $\ds \lim_{t\to 0^+} \varphi(t)=0$, then for $1\le p<\infty$,  and $0<q\le \min\{p,2\}$,
\begin{align}\label{18-11_3}
	\egv{\mathcal{N}^0_{\varphi,p,q}(\rd)}(t) \,\sim\,  t^{-\frac1p}, \quad t>0.
\end{align} 
\item
  If $\ds \lim_{t\to 0^+} \varphi(t)>0$, then for $1<q\le 2$, there are constants $c_2>c_1>0$ such that 
  \begin{equation}\label{Case-3b}
c_1\    |\log t|^{\frac{1}{q'}} \leq \egv{\mathcal{N}^0_{\varphi,p,q}(\rd)}(t) \leq c_2\ |\log t|^{\frac{2}{q'}},\quad 0<t<t_0,
  \end{equation}
where $t_0=t_0(p,q)$.
\end{enumerate}
 \end{proposition}

 \begin{proof}
   {\em Step 1}.~   Let us first briefly discuss the conditions \eqref{cond1} and $\MB(\rd)\subset \Lloc(\rd)$ in case of $s=0$. In case (i), when $\ds \lim_{t\to 0^+} \varphi(t)=0$, then Theorem~\ref{emb-L1loc}(ii) gives that for $p\geq 1$,  $\mathcal{N}^0_{\varphi,p,q}(\rd) \subset \Lloc(\rd)$ if, and only if, $0<q\leq \min\{p,2\}$, while in case of $0<p<1$ condition \eqref{suff-cond-c} with $s=0$ is never satisfied. So we restrict ourselves to the case $p\geq 1$, $0<q\leq \min\{p,2\}$ in (i). Concerning (ii), Theorem~\ref{emb-L1loc}(i) tells us that we only need to consider $0<q\leq 2$. Moreover, \eqref{cond1} requires $q'<\infty$, that is, $q>1$. So in (ii) we are left to consider $1<q\leq 2$.\\

   {\em Step 2}.~ We deal with (i) and assume $p\geq 1$, $0<q\leq \min\{p,2\}$. Concerning the upper estimate in \eqref{18-11_3} we benefit from \eqref{MB-in-B} (with $s=0$)  and the corresponding result in \eqref{egv-B-subcritical}, to conclude that
   \begin{equation*}
   \egv{\mathcal{N}^0_{\varphi,p,q}(\rd)}(t) \leq \ c_1 \ \egv{B^0_{p,q}(\rd)}(t)\ \leq c_2\ t^{-\frac1p}, \quad t>0.
   \end{equation*}
To prove the estimate from below we adopt the method used  in the proof  of Theorem 3.1 in  \cite{hs13}, cf. also  the proof of Theorem 3.2  in \cite{hs12}. 
For  $0<j$ we put 
$$k_j = \whole{2^{\nd j}\varphi(2^{-j})^p},$$
where  $\whole{x}=\max\{l\in \mathbb{Z}: l \leq x\}$. Then $1\le k_j < 2^{\nd j}$ and $k_j\rightarrow \infty$ if $j\rightarrow \infty $ since $2^{\nd j/p}\varphi(2^{-j})\rightarrow \infty $, recall assumption \eqref{case-C}.  
For any  $j>0$  we define a finite sequence $\lambda_{j,m_k}=\varphi(2^{-j})^{-1}$, with  $m_k=(4^k,0,\ldots, 0)$, $k=1,\ldots, k_j$.  
Let   $\psi$ be the compactly supported $C^{\infty}$ function on
$\rd$ defined by
\begin{equation*}
	\psi(x)=\widetilde{\psi}(2x-(1,\ldots,1)) - \widetilde{\psi}(2x+(1,\ldots,1)),
\end{equation*}
where 
\begin{equation} \label{Psi}
	\widetilde{\psi}(x)=e^{-1/(1-|x|^2)} \quad \text{if} \quad |x|<1\quad
	\text{and} \quad \widetilde{\psi}(x)=0 \quad \text{if} \quad |x|\geq 1.
\end{equation}
The function $\psi$ is supported in the unit ball and satisfies the first moment conditions. So any function $\psi_{j}(x)=\psi(2^jx)$ is an atom satisfying the first moment conditions. 
We consider the functions 
\begin{equation}\nonumber
	f_j(x) =  
	\sum_{k=1}^{k_j} \lambda_{j,m_k}\psi_{j,m_k}(x)  .
\end{equation}
where $\psi_{j,m_k}(x)$ is a suitable translation of $\psi_{j}$ such that the  atom $\psi_{j,m_k}$ is centred at $Q_{j,m_k}$. The functions $f_j$, $j\in\nat$, belong to $\mathcal{N}^0_{\varphi,p, q}(\rd)$ and their norms are uniformly bounded,  $\|f_j|\mathcal{N}^0_{\varphi,p, q}(\rd)\|  \le c$, since 
\begin{align*}
	\mathop{\sup_{\nu: 0\le\nu \leq j}}_{k\in \zd}\!  \varphi(2^{-\nu}) \,2^{(\nu-j)\frac{\nd}{p}}\Big(\!\! \mathop{\sum_{m\in\zd:}}_{Q_{j,m}\subset Q_{\nu,k}}\!\!|\lambda_{j,m}|^p\Big)^{\frac 1 p}= 1  
	\intertext{and}
	2^{-j\frac{\nd}{p}} 
	\Big(\!\! \mathop{\sum_{m\in\zd}}
	\!\!|\lambda_{j,m}|^p\Big)^{\frac 1 p} = 
	2^{-j\frac{\nd}{p}}  k_j^{\frac{1}{p}} \varphi(2^{-j})^{-1} \le 1  .   \nonumber
\end{align*} 
Each function $f_j$, $j\in\nat$, takes the value ${c_1}\varphi(2^{-j})^{-1}$, $c_1\in (0,1)$, on a set of measure ${c_2}2^{-j\nd}k_j\sim \varphi(2^{-j})^p$. In consequence, 
\begin{equation}\label{fj2}
	f_j^*(c \varphi(2^{-j})^p)\ge C \varphi(2^{-j})^{-1}, \qquad j=0,1,2,\ldots,
\end{equation}
and  the constants {$c$,} $C$  used in the last inequality {are} independent of $j\in\nat$. The definition  of the class $\Gp$ and our assumption in (i) imply that $\varphi(2^{-j})\rightarrow 0$ if $j\rightarrow \infty$  and 
$\varphi(2^{-j-1})\le \varphi(2^{-j}) \le 2^{\nd/p} \varphi(2^{-j-1})$.    
So, now  we arrive at the  desired estimate from below  
\[
C t^{-\frac1p} \le  \egv{\mathcal{N}^0_{\varphi,p, q}}(t) , \qquad 0<t<1,
\]
if $1\le p< \infty$ and $0< q\le \min\{p,2\}$. \\
{\em Step 3}.~ We finally consider (ii) and assume $1<q\leq 2$ now. The main clue will be to prove that now
\begin{equation}\label{rem-leszek}
\mathcal{N}^0_{\varphi,p,q}(\rd) = B^0_{\infty,q}(\rd) \cap B^0_{p,q}(\rd) \subset \Lloc(\rd).
\end{equation}
If this is verified, the lower estimates  follows immediately: since $B^{\frac{\nd}{p}}_{p,q}(\rd) \hookrightarrow B^0_{\infty,q}(\rd) \cap B^0_{p,q}(\rd) = \mathcal{N}^0_{\varphi,p,q}(\rd)$, it is granted by \eqref{egv-B-critical}. To prove the upper estimate we use norms in Besov spaces $B^s_{p,q}(\rd)$ given by a smooth dyadic  resolution of unity $\{\eta_j \}_{j=0}^\infty$, 
\begin{align*}
	0\le \eta_0\le 1, \quad \eta_0(x) = 
	\begin{cases}
		1\quad & \text{if}\; |x|\le 1 \\
	    0\quad & \text{if}\;  |x|\ge \frac{3}{2}  
	\end{cases} \,\text{and}\qquad 
	    \eta_j(x)= \eta_0(2^{-j}x) - \eta_0(2^{-j+1}x). 
\end{align*}
If $r > \max(p,2)$ and $\theta= \frac{p}{r}$ then
\begin{equation}\label{log1}
	\|f| B^0_{r,q}(\rd)\| \le  \|f| B^0_{p,q}(\rd)\|^\theta \, \|f| B^0_{\infty,q}(\rd)\|^{1-\theta} \le \|f| \mathcal{N}^0_{\varphi,p,q}(\rd)\|. 
\end{equation}
One can easily see that 
\begin{align}
	\|f| L_{r}(\rd)\| &\le  \|f| B^0_{r,1}(\rd)\| \label{log2}
	\intertext{and}
	\|f| L_{r}(\rd)\| &\le  c_r\|f| F^0_{r,2}(\rd)\| \le c_r \|f| B^0_{r,2}(\rd)\|, \label{log3}
\end{align}
for some constant $c_r>0$ depending on $r$. Moreover it follows from properties of Banach-valued singular integral operators that $c_r= c\, r$ where $c>0$ is a constant depending of $d$ and the family  $\{\eta_j\}$ but independent  of $r$, cf. \cite{Graf} Theorem 4.6.1 and Theorem 5.1.2. Interpolating between \eqref{log2} and \eqref{log3} we get 
\begin{equation}
	\|f| L_{r}(\rd)\| \le C  r^{\frac{2}{q'}}\|f| B^0_{r,q}(\rd)\| , \qquad 1<q<2 \label{log4}
\end{equation}
As a consequence of \eqref{log1} and \eqref{log4} we get 
\[
\egv{\mathcal{N}^0_{\varphi,p,q}(\rd)}(t) \leq \ c_r \ \egv{B^0_{r,q}(\rd)}(t)\ \leq c r^{\frac{2}{q'}} \ t^{-\frac1r},
\]

 Now let $0<t<\frac12$ be fixed and choose $r_0 = r(t)$ in an optimal way, that is, such that $r_0^{2/q'} t^{-1/r_0}$ becomes minimal. Straightforward calculation leads to $r_0 = q' |\log t|/2$, such that $r_0^{2/q'} t^{-1/r_0} \sim |\log t|^{2/q'}$ is the minimum, taken in $r_0\geq \max({2},p)$ if $t< t_0=e^{-2p/q'}$. Hence we arrive at
   \[
    \egv{\mathcal{N}^0_{\varphi,p,q}(\rd)}(t) \leq \ c_2 \  |\log t|^{2/q'}, \quad 0<t<t_0,
   \]
   as desired.

It remains to verify \eqref{rem-leszek}. Note that \eqref{case-C} implies that   $2^{\nu\frac{\nd}{p}}\varphi(2^{-\nu}) \sim 1$ if $\nu\le 0$, so
	\begin{align}\label{22.06-2}
		&	\mathop{\sup_{\nu: \nu \leq j}}_{k\in \zd}\!  \varphi(2^{-\nu}) \,2^{(\nu-j)\frac{\nd}{p} }\Big(\!\! \mathop{\sum_{m\in\zd:}}_{Q_{j,m}\subset Q_{\nu,k}}\!\!|\lambda_{j,m}|^p\Big)^{\frac 1 p}  \sim \\
		&	\max\, \bigg\{	\mathop{\sup_{\nu: 0\le\nu \leq j}}_{k\in \zd}\!  \varphi(2^{-\nu}) \,2^{(\nu-j)\frac{\nd}{p}}\Big(\!\! \mathop{\sum_{m\in\zd:}}_{Q_{j,m}\subset Q_{\nu,k}}\!\!|\lambda_{j,m}|^p\Big)^{\frac 1 p},
		2^{-j\frac{\nd}{p}} 
		\Big(\!\! \mathop{\sum_{m\in\zd}}
		\!\!|\lambda_{j,m}|^p\Big)^{\frac 1 p}\bigg\} . 
		\nonumber
	\end{align}
Hence, by our assumption in (ii), $\varphi(2^{-\nu}) \sim 1$, $0\leq \nu\leq j$, thus  
     \begin{align}
     \|\lambda | n^0_{\varphi,p,q}\| & \sim \left(\sum_{j=0}^\infty \max\, \bigg\{	\mathop{\sup_{\nu: 0\le\nu \leq j}}_{k\in \zd}\!  \varphi(2^{-\nu}) \,2^{(\nu-j)\frac{\nd}{p}}\Big(\!\! \mathop{\sum_{m\in\zd:}}_{Q_{j,m}\subset Q_{\nu,k}}\!\!|\lambda_{j,m}|^p\Big)^{\frac 1 p},
     	2^{-j\frac{\nd}{p}} 
     	\Big(\!\! \mathop{\sum_{m\in\zd}}
     	\!\!|\lambda_{j,m}|^p\Big)^{\frac 1 p}\bigg\}^q \right)^{1/q} \notag\\
        & \sim \left(\sum_{j=0}^\infty \max\, \bigg\{	\mathop{\sup_{\nu: 0\le\nu \leq j}}_{k\in \zd}\!  2^{(\nu-j)\frac{\nd}{p}}\Big(\!\! \mathop{\sum_{m\in\zd:}}_{Q_{j,m}\subset Q_{\nu,k}}\!\!|\lambda_{j,m}|^p\Big)^{\frac 1 p},
     	2^{-j\frac{\nd}{p}} 
     	\Big(\!\! \mathop{\sum_{m\in\zd}}
     	\!\!|\lambda_{j,m}|^p\Big)^{\frac 1 p}\bigg\}^q \right)^{1/q} \notag\\
        & \sim \left(\sum_{j=0}^\infty \max\, \bigg\{ \sup_{m\in\zd} |\lambda_ {j,m}|^q,    	2^{-j\frac{\nd}{p}q} 
     	\Big(\!\! \mathop{\sum_{m\in\zd}}
     	\!\!|\lambda_{j,m}|^p\Big)^{\frac q p}\bigg\} \right)^{1/q} \notag\\
& \sim \max\{ \|\lambda | b^0_{\infty,q}\|, \|\lambda | b^0_{p,q}\|\} \sim \|\lambda | b^0_{\infty,q} \cap b^0_{p,q}\|.
        \label{2010}        
     \end{align} 
     This leads to  \eqref{rem-leszek}.
 \end{proof}

 \begin{remark}
   Note that for $B^0_{\infty,q}(\rd)$ no growth envelope result is known, but presumably it is infinite in any $t>0$, while in case of a bounded domain (or the periodic case) Seeger and Trebels proved in \cite{SeeT} that $\egv{B^0_{\infty,q}(\Omega)}(t)\sim |\log t|^{1/q'}$. Here we have the intersection $B^0_{\infty,q}(\rd) \cap B^0_{p,q}(\rd)$ where apparently $B^0_{\infty,q}(\rd)$ determines the (local) growth envelope, while $B^0_{p,q}(\rd)$ ensures (global) finiteness. We have no precise result of the $\log$-exponent in \eqref{Case-3b} yet, but find the apparently logarithmic behaviour remarkable. It represents a new phenomenon not visible in the known cases like in Besov or Triebel-Lizorkin spaces yet. This is due to the assumed behaviour of $\varphi$ near $0$, that is, $\lim_{t\to 0^+} \varphi(t)>0$, which is plainly very much different from all `classical' example functions $\varphi$ considered before. It substantiates again the importance to study the more general function $\varphi\in\Gp$ here.
 \end{remark}

 Now we deal with the remaining case. Recall that for $\varphi\in\Gp$ we may always assume that $\varphi$ is a continuous strictly increasing function, cf. \cite[Prop. 23, p.~25]{FHS-MS-2}.  We denote the function inverse to $\varphi$  by $\varphi^{-1}$.

\begin{proposition}\label{P-case-4}
   	Let $0<p<\infty$, $0<q\leq \infty$,  $\varphi\in \Gp$ with $\varphi(1)=1$, and $s>0$ so that {$\MB(\rd) \subset \Lloc(\rd)$}. 
	Assume   conditions \eqref{cond1} and \eqref{case-C} to hold. Then $\lim_{t\to 0^+} \varphi(t)=0$.
\begin{enumerate}[{\bfseries\upshape(i)}]
\item
  If, additionally, 
$$
s> \sigma_p \quad  \text{or}  \quad s = \nd\left(\frac1p-1\right)>0, \quad 0<p<1, \quad \text{with} \quad 0<q\leq 2,
$$
then there exists $C>0$ such that
\begin{equation}\label{estimate-above}
  \egv{\MB(\rd)}(t)  
		 \le C\,   
	\begin{cases} t^{-\frac1p+\frac{s}{\nd}} , & \text{if} \;\; s<\frac{\nd}{p},   \; \text{or}\\[1ex]
|\log t|^{\frac{1}{q'}}  , & \text{if} \;\;  s=\frac{\nd}{p} \;\; \text{and} \;\;1<q\leq\infty. \end{cases}  
\end{equation}
\item
There exists some $c>0$ such that
\begin{align}\label{31-10_1}
	\egv{\MB(\rd)}(t)  \geq  \,	c\, \varphi^{-1}\left(t^{\frac{1}{p}}\right)^s t^{-\frac{1}{p}},\quad 0<t<1. 
		\end{align}
\item
If, in addition, there  exist positive constants $\alpha>0$, $\beta>0$ and $\gamma>0$,  such that 
\begin{equation}\label{cond2a}
	0<\beta\le\liminf_{j\rightarrow \infty} \varphi(2^{-j}) 2^{j\alpha} \le  \limsup_{j\rightarrow \infty} \varphi(2^{-j}) 2^{j\alpha}\le \gamma<\infty,
\end{equation}
and $q$ satisfies 
\begin{equation}\label{cond3}
	\frac{1}{q}\ge \max\left\{\frac{1}{2},\frac{1-s\alpha^{-1}}{p}\right\},
\end{equation}
then {for $\alpha (1-p)_+< s<\alpha$},
 	\begin{align}\label{31-10_2}
 	\egv{\MB(\rd)}(t) \,\sim \,
 	 t^{-\frac{1}{p}} \, \varphi^{-1}\left(t^{\frac{1}{p}}\right)^s.    
 \end{align}  
\end{enumerate}
\end{proposition}

\begin{proof}
Note first that  \eqref{cond1} implies $\lim_{t \to 0^+} \varphi(t) = 0$ when $s>0$. \\
  
  \emph{Step 1.}~ The estimate  \eqref{estimate-above} in (i) is a direct consequence of \eqref{case-C} which implies \eqref{MB-in-B}, Proposition \ref{properties:EG}(i),  \eqref{egv-B-subcritical} and \eqref{egv-B-critical}, complemented by \cite[Prop.~8.12, 8.14]{Ha-crc}. \\

  \emph{Step 2.}~ We deal with the estimate from below in (ii). Let $\tilde{\psi}$ be the function given by \eqref{Psi}. We define, for $j\in\nat$,
		\begin{align*}
		f_j(x) &=2^{-js}\varphi(2^{-j})^{-1} \tilde{\psi}(2^jx) ,\\
		g_j(x) &=\sum_{k=1}^{k_j} f_j(x- m_k), 
	\end{align*}
	where  $ m_k=(4^{k},0,\ldots ,0 )$  and $k_j= \whole{2^{j\nd}\varphi\left(2^{-j}\right)^p}$. 
	We have 
	\begin{align*}
		\|g_j|\MB(\rd) \| &\sim  2^{js}\max\, \bigg\{	\varphi(2^{-j})2^{-js}\varphi\left(2^{-j}\right)^{-1} , 2^{-j\frac{\nd}{p}}{k_j}^{1/p} 2^{-js}\varphi\left(2^{-j}\right)^{-1}\bigg\} \\ &\sim \max\, \left\{ 1, {k_j}^{1/p}2^{-j\frac{\nd}{p}} \varphi\left(2^{-j}\right)^{-1}\right\} \sim 1.
	\end{align*}
In consequence, 
	\begin{equation}\label{hj}
		g_j^*\left({c}\varphi(2^{-j}\right)^p)\ge C 2^{-js}\varphi\left(2^{-j}\right)^{-1}= C
	\left(\varphi^{-1}\left(\left(\varphi(2^{-j})^p\right)^\frac{1}{p}\right) \right)^s\left(\varphi\left(2^{-j}\right)^{p}\right)^{-1/p} , \qquad j\in\no.
      \end{equation}
      Moreover, $\varphi$ is a strictly increasing function that satisfies the doubling condition,  i.e., $\varphi(t)\le\varphi(2t)\le 2^{\nd/p} \varphi(t)$,  and $\lim_{t\to 0^+}\varphi(t)=0$. Therefore 
	\begin{equation*}
		g_j^*(t)\ge C \ t^{-1/p} \ \varphi^{-1}\left(t^\frac{1}{p}\right)^s , \quad t>0, 
	\end{equation*}
	and 
	\[ 
 C \ t^{-1/p} \varphi^{-1}\left(t^\frac{1}{p}\right)^s  \le  \egv{\mathcal{N}^s_{\varphi,p, q}}(t) , \qquad 0<t<1.
 \]
 {\em Step 3}.~ 
 Now we assume that \eqref{cond2a} holds. Please note that $\alpha<\frac{\nd}{p}$. We choose {$u$} such that $\frac{\nd}{u}=\alpha$. Then {$p<u<\infty$} and  the function 
  \begin{equation}\label{phi_up}
  	\widetilde{\varphi}(t)=\begin{cases}
  		t^{\nd/u}\qquad\text{if} \qquad 0<t\le 1,\\
  		t^{\nd/p}\qquad\text{if} \qquad t>1, 
  	\end{cases}
  \end{equation} 
satisfies our assumption. If  $0< s< \frac{\nd}{u}=\alpha$, then we can find $r$ with $p\le r< \infty$, such that $ \frac{1}{r}=\frac{1}{p}- \frac{su}{\nd p}$, resulting in  
\[ \left\{2^{-js}\widetilde{\varphi}(2^{-j})^{\frac{p}{r}-1}
\right\}_{j\in\nat}\in 
\ell_\infty\qquad \text{if}\qquad \frac{1}{r}=\frac{1}{p}- \frac{su}{\nd p}.\] 
So { Corollary~\ref{m_b_emb} }  
implies
\begin{equation}\label{emb-Lr}
	\mathcal{N}^s_{\widetilde{\varphi},p,q}(\rd)\hookrightarrow B^0_{r,q}(\rd).
\end{equation}
Now the upper estimates in  \eqref{31-10_2} follows from {\eqref{egv-B-subcritical}} 
{for $q\le \min(2,r)$ and $r\geq 1$}, that is, $q$ satisfying \eqref{cond3} and {$\alpha (1-p)_+<s<\alpha$}, together with 
\[	\mathcal{N}^s_{\varphi,p,q}(\rd)\hookrightarrow 	\mathcal{N}^s_{\widetilde{\varphi},p,q}(\rd),   \] 
{which follows from the lower estimate in \eqref{cond2a} and Theorem~\ref{main}.}
The upper estimate in \eqref{cond2a}  gives us
\[ 
	t^{-\frac{1}{p}+\frac{su}{\nd p}} \le c \ t^{-\frac{1}{p}} \varphi^{-1}\left(t^{\frac{1}{p}}\right)^s.
\]
\end{proof}

\begin{examples}
  Let us point out that condition \eqref{cond2a} is neither trivially satisfied for all $\varphi\in\Gp$, nor impossible. For the latter argument, let $0<s<\alpha<\frac{\nd}{p}$ and  take
  \[
  \varphi_\alpha (t) = \begin{cases} t^\alpha, & t\leq 1, \\ t^{\frac{\nd}{p}}, & t>1,\end{cases} \]
  which satisfies  \eqref{cond1}, \eqref{case-C}, and \eqref{cond2a}.
  On the other hand, the function $\varphi^\ast(t)$, defined by
  \[\varphi^\ast(t) = \begin{cases} 0, & t=0, \\ -t \ln t, & 0<t<e^{-1}, \\ t, & t\geq e^{-1}, \end{cases}
  \]
  is a continuous strictly increasing function belonging to $\Gp$ with $p=\nd$. Moreover,
  \[\lim_{t\rightarrow 0^+} t^{-\alpha}\varphi^\ast(t) = \begin{cases}  \infty, & \alpha \geq 1, \\ 0, & 0<\alpha<1. \end{cases} \]
Then $\varphi^\ast$ satisfies \eqref{cond1} for any $s<1 = \frac{\nd}{p}$ as well as \eqref{case-C}, but not \eqref{cond2a}.  
\end{examples}

We conclude the paper with two examples.

  \begin{example}\label{phi-uv-rev}
    Let $0<p\leq u,v<\infty$, $s\in\rr$, $0<q\leq\infty$, and consider the function  $\varphi_{u,v}$  from Example~\ref{exm-Gp}(i). We consider the growth envelope function $\mathcal{N}^s_{\varphi_{u,v},p,q}(\rd)$. According to our findings in Propositions~\ref{P-case-1}, \ref{P-case-2}, \ref{P-case-3} and \ref{P-case-4} the limits
    \[
    \lim_{t\rightarrow 0^+} \varphi_{u,v}(t), \qquad  \lim_{t\rightarrow 0^+} \varphi_{u,v}(t)t^{-\frac{\nd}{p}}, \qquad \text{and}\quad \lim_{t\rightarrow \infty} \varphi(t) t^{-\frac{\nd}{p}}
    \]
    are important. So let us first consider these. Plainly, we always have
    \[ \lim_{t\rightarrow 0^+} \varphi_{u,v}(t) =  \lim_{t\rightarrow 0^+} t^{\frac{\nd}{u}} = 0,    \]
    since $0<u<\infty$, which immediately excludes the new case dealt with Proposition~\ref{P-case-3}(ii). As for the second and third limit, we have
    \[
    \lim_{t\rightarrow 0^+} \varphi_{u,v}(t)t^{-\frac{\nd}{p}} =  \begin{cases} 1, & u=p, \\ \infty, & u>p, \end{cases}\qquad \text{and}\qquad
    \lim_{t\rightarrow \infty} \varphi_{u,v}(t)t^{-\frac{\nd}{p}}  = \begin{cases} 1, & v=p, \\ 0, & v>p. \end{cases}
    \]
    Concerning the question whether $\mathcal{N}^s_{\varphi_{u,v},p,q}(\rd)\not\hookrightarrow L_\infty(\rd)$, condition \eqref{cond1} reads as
    \[
    \begin{cases} s<\frac{\nd}{u}, & 0<q\leq\infty,\\
      s=\frac{\nd}{u}, & 1<q\leq\infty.\end{cases}
    \]
    Obviously, if $s>\frac{\nd}{u}$ or $s=\frac{\nd}{u}$ and $0<q\leq 1$, then $\egv{\mathcal{N}^s_{\varphi_{u,v},p,q}(\rd)}(t) \sim 1$, $t>0$, recall Corollary~\ref{Cor-SM-1}(i). 
    Finally, to ensure $ \mathcal{N}^s_{\varphi_{u,v},p,q}(\rd)\subset\Lloc(\rd)$ Theorem~\ref{emb-L1loc}(ii), (iii) yield for $p\geq 1$ the assumptions \eqref{suff2-Lloc2}, and for $0<p<1$,
    \[
\begin{cases} s> \nd\frac{1-p}{u}, & 0<q\leq\infty, \\ s= \nd\frac{1-p}{u}, & 0<q\leq 1.\end{cases}
    \]
    So we are left to consider the cases
    \begin{align}\label{cond-suv-1}
1\leq p<\infty &\qquad\text{and}\qquad \begin{cases} s=0, & 0<q\leq\min(p,2), \\ 0<s<\frac{\nd}{u}, & 0<q\leq\infty, \\ s=\frac{\nd}{u}, & 1<q\leq\infty,\end{cases}
\intertext{and}
0<p<1 &\qquad\text{and}\qquad \begin{cases} s= \nd\frac{1-p}{u}, & 0<q\leq 1, \\  \nd\frac{1-p}{u}<s<\frac{\nd}{u}, & 0<q\leq\infty, \\ s=\frac{\nd}{u}, & 1<q\leq\infty.\end{cases}\label{cond-suv-2}
    \end{align}
    Let \eqref{cond-suv-1} or \eqref{cond-suv-2} be satisfied in the sequel.
    \begin{enumerate}[{\bfseries\upshape(i)}]
    \item
      Assume $v>p$. Then
      \[\egv{\mathcal{N}^s_{\varphi_{u,v},p,q}(\rd)}(t) = \infty, \quad t>0.\]
      This corresponds to Proposition~\ref{P-case-1}.
    \item
      If $u=v=p$, then $\mathcal{N}^s_{\varphi_{u,v},p,q}(\rd) = \mathcal{N}^s_{p,p,q}(\rd) = \B(\rd)$ and the result is recalled in \eqref{egv-B-subcritical} and \eqref{egv-B-critical}. This result obviously corresponds to the more general situation considered in Proposition~\ref{P-case-2}.
    \item
      Assume $v=p<u$ and $s=0$ with $p\geq 1$ and $0<q\leq\min(p,2)$, according to \eqref{cond-suv-1}. Then by Proposition~\ref{P-case-3}(i),
       \[\egv{\mathcal{N}^0_{\varphi_{u,p},p,q}(\rd)}(t) \sim t^{-\frac1p}, \quad t>0.\]
\item
  Assume $v=p<u$ and $s>0$ satisfying \eqref{cond-suv-1} or \eqref{cond-suv-2}. Let first $p\geq 1$ and $0<s<\frac{\nd}{u}$, then condition \eqref{cond2a} is satisfied with $\alpha= \frac{d}{u}$. Then we get for $q\leq \min\{2, \frac{p\nd}{\nd-su}\}$ that
  \[
  \egv{\mathcal{N}^s_{\varphi_{u,p},p,q}(\rd)}(t) \sim t^{-\frac1p+\frac{us}{p\nd}} = t^{\frac{u}{p}\left(\frac{s}{\nd}-\frac1u\right)}, \quad t>0.
  \]
  This is obviously better than the upper estimate $t^{-\frac1p+\frac{s}{\nd}}$ in \eqref{estimate-above}, since $u>p$ now. If $0<p<1$, we have to assume, that $\nd \frac{1-p}{u}<s<\frac{\nd}{s}$, but the result remains the same. If $s=\nd\frac{1-p}{u}>0$, then it can be written as
  \[
  \egv{\mathcal{N}^s_{\varphi_{u,p},p,q}(\rd)}(t) \sim t^{-1}, \quad t>0,
  \]
  assuming that $0<q\leq 1$.
  
  In case of $s=\frac{\nd}{u}$, $q>1$, $0<p<\infty$, the choice of $\alpha>s = \frac{\nd}{u}$ satisfying \eqref{cond2a} is impossible, so we are left with the estimate from  above by \eqref{estimate-above},
  \[
\egv{\mathcal{N}^{\nd/u}_{\varphi_{u,p},p,q}(\rd)}(t) \leq c\ t^{-\frac1p+\frac{1}{u}}, \quad t>0, 
\]
since $u>p$ implies $s=\frac{\nd}{u} < \frac{\nd}{p}$. As for the estimate from below, \eqref{31-10_1} would lead to lower estimate $\egv{\mathcal{N}^{\nd/u}_{\varphi_{u,p},p,q}(\rd)}(t) \geq c$ only, but instead we use the embedding \eqref{ex-last-2} together with Proposition~\ref{properties:EG} and \eqref{egv-B-critical} to get 
\[
\egv{\mathcal{N}^{\nd/u}_{\varphi_{u,p},p,q}(\rd)}(t) \geq c\ |\log t|^{1/q'}, \quad 0<t<\frac12. 
\]
This is the only case when we do not have a sharp result yet.      
\end{enumerate}
	\end{example}
    
  \begin{example}
    Let $0<p<\infty$. At last we return to Example~\ref{exm-Gp}(ii),
    \[\varphi^+(t) = \max\left(t^\frac{\nd}{p},1\right) \in \Gp, \]
    which corresponds to $\varphi_{u,v}$ studied above, but now with $u=\infty$. Then
    \[
    \lim_{t\to 0^+} \varphi^+(t) = 1, \qquad  \lim_{t\rightarrow 0^+} \varphi^+(t)t^{-\frac{\nd}{p}}=\infty, \qquad \text{and}\quad \lim_{t\rightarrow \infty} \varphi^+(t) t^{-\frac{\nd}{p}} = 1.
    \]
    Clearly, \eqref{cond1} provides that $\mathcal{N}^s_{\varphi^+,p,q}(\rd) \not\hookrightarrow L_\infty (\rd)$ if $s<0$ or $s=0$ and $q>1$. Moreover, by Theorem~\ref{emb-L1loc}(i), we need $s>0$ or $s=0$ and $0<q\leq 2$ for $ \mathcal{N}^s_{\varphi^+,p,q}(\rd)\subset\Lloc(\rd)$. So we are left to consider $s=0$ and $1<q\leq 2$ only. Then Proposition~\ref{P-case-3}(ii) results in
    \[
    c_1\ |\log t|^{\frac{1}{q'}} \leq \egv{\mathcal{N}^0_{\varphi^+,p,q}(\rd)}(t) \leq c_2\ |\log t|^{\frac{2}{q'}}. 
    \]
In all other cases the growth envelope for     $ \mathcal{N}^s_{\varphi^+,p,q}(\rd)$ is either a constant, or does not exist. As already mentioned above, this `new' case seems remarkable to us, apart from the fact that the exponent is not quite sharp yet.   
  \end{example}

\bibliographystyle{alpha}	
\def\cprime{$'$}

\bigskip~\bigskip~

{\small
\noindent
Dorothee D. Haroske\\
Institute of Mathematics \\
Friedrich Schiller University Jena\\
07737 Jena\\
Germany\\
{\tt dorothee.haroske@uni-jena.de}\\[4ex]
%
Susana D. Moura\\
University of Coimbra\\
CMUC, Department of Mathematics\\
Largo D. Dinis, 3000-143 Coimbra\\
Portugal\\
\texttt{smpsd@mat.uc.pt}\\[4ex]
%
%
Leszek Skrzypczak\\
Faculty of Mathematics \& Computer Science\\
Adam  Mickiewicz University\\
ul. Uniwersytetu Pozna\'nskiego 4\\
61-614 Pozna\'n\\
Poland\\
{\tt lskrzyp@amu.edu.pl}
}
}

\end{document}